\documentclass[smallextended,numbook,runningheads]{svjour3}     
\smartqed  
\usepackage{graphicx}
\usepackage{amsmath}
\usepackage{epstopdf} 
\usepackage{mathptmx}      

\usepackage{epsfig,bm,epstopdf,graphicx}
\usepackage{amssymb,epsfig,mathrsfs,mathpazo}
\usepackage[multiple]{footmisc}
\usepackage[font=footnotesize,labelfont=bf]{caption}
\usepackage{color}

\newcommand{\be}{\begin{equation}}
\newcommand{\ee}{\end{equation}}

\newcommand{\R}{\mathbb R}
\newcommand{\C}{\mathbb C}

\newcommand{\N}{\mathbb N}

\newcommand{\cL}{\mathcal L}

\newcommand{\Lis}{\cL_{\mathrm{iso}}}
\newcommand{\identity}{\mathrm{Id}}

\newcommand{\res}[1]{{\color{black}{#1}}}
\newcommand{\revision}[1]{{\color{black}{#1}}}
\newcommand{\revv}[1]{{\color{black}{#1}}}

\DeclareMathOperator{\argmin}{argmin}
\DeclareMathOperator{\Erfc}{Erfc}

\journalname{BIT}
\begin{document}

\title{Efficient numerical approximation of a non-regular Fokker--Planck equation associated with first-passage time distributions\thanks{GG was supported by the Austrian Science Fund (FWF) under grant J4379-N.}}

\titlerunning{Numerical approximation of a Fokker--Planck equation}        

\author{Udo Boehm \and Sonja Cox \and Gregor Gantner \and Rob Stevenson} 

\authorrunning{U.\ Boehm \and S.\ Cox \and G. Gantner \and R. Stevenson} 

\institute{U.\ Boehm \at Department of Psychology\\ 
University of Amsterdam\\  
PO Box 15906\\
1001 NK Amsterdam\\ 
The Netherlands\\
\email{u.bohm@uva.nl}
\and 
S.\ Cox, G.\ Gantner, R.\ Stevenson \at 
Korteweg--de Vries (KdV) Institute for Mathematics\\
University of Amsterdam\\
PO Box 94248\\
1090 GE Amsterdam\\
The Netherlands\\
\email{s.g.cox@uva.nl; g.gantner@uva.nl; r.p.stevenson@uva.nl}
}

\date{Received: date / Accepted: date}

\maketitle

\begin{abstract}

In neuroscience, the distribution of a decision time is modelled by means of a one-dimensional Fokker--Planck equation with time-dependent boundaries and space-time-dependent drift. 
Efficient approximation of the solution to this equation is required, e.g., for model evaluation and parameter fitting. However, the prescribed boundary conditions lead to a strong singularity and thus to slow convergence of numerical approximations. In this article we demonstrate that the solution can be related to the solution of a parabolic PDE on a rectangular space-time domain with homogeneous initial and boundary conditions by transformation and subtraction of a known function. We verify that the solution of the new PDE is indeed more regular than the solution of the original PDE and proceed to discretize the new PDE using a space-time minimal residual method. We also demonstrate that the solution depends analytically on the parameters determining the boundaries as well as the drift. This justifies the use of a sparse tensor product interpolation method to approximate the PDE solution for various parameter ranges. The predicted convergence rates of the minimal residual method and that of the interpolation method are supported by numerical simulations.




\end{abstract}

\section{Introduction}\label{sec:introduction}
In 1978 Ratcliff~\cite{Ratcliff1978} introduced a model for binary decision processes based on diffusion processes. This model turned out to agree well with experimental data; Gold and Shadlen~\cite{Gold2001} provides a neurophysiological explanation for its success. Indeed, the solution $(X_t)_{t\geq 0}$ of a one-dimensional stochastic differential equation is assumed to describe the difference in activity of two competing neuron populations. At time $t = 0$, the value $X_0 = x_0\in \R$ represents the resting-state activity of the neuron populations. A decision is triggered when $(X_t)_{t\geq 0}$ first reaches one of two (possibly time-dependent) critical values $\alpha$ or $\beta$, each reflecting an outcome of the decision process.\par 

In a typical decision experiment, scientists can only measure the decision time and outcome. Parameter fitting thus requires access to the decision time distributions, which are rarely known explicitely. 
Ad hoc numerical simulations are costly whence efficient simulation methods are much sought-after~\cite{Hawkins2015,Fengler2020}.\par  

\res{In this article we} extend and improve a simulation method introduced in~\cite{Voss2008}, which is based on the Fokker--Planck equation associated to the decision time.  
\res{In particular, this article may be viewed as the theoretical counterpart of our publication~\cite{BCGS21}, which is aimed at the neuroscientific community.\par Linking the first hitting time of a stochastic differential equation to a Fokker--Planck \revision{equation} is a well-known approach that has also been applied in e.g.\ astrophysics~\cite{Chandrasekhar:1943} and cell biology~\cite{HolcmanSchuss:2015}; for an overview see~\cite{ArtimeEtAl:2018}. In particular, although we only consider examples arising from neuroscience, the simulation method we introduce is also relevant for other applications.\par} 
To explain the \res{Fokker--Planck based} approach consider the following stochastic differential equation:
\begin{equation}
 dX_t^{y} = \mu(t,X_t^{y}) \,dt + \sigma \,dW_t \quad t\in [0,\infty),\, X_0^{y}=y.
\end{equation}
Here $(W_t)_{t\in [0,\infty)}$ is a Brownian motion, $\sigma\in (0,\infty)$ is the diffusion parameter, $\mu \in C([0,\infty)\times \R)$ is the (time- and state-dependent) drift and $y\in \R$ is the initial value. Let $\alpha,\beta\in C^1([0,\infty))$ satisfy $\alpha\leq\beta$, and for all $y\in [\alpha(0),\beta(0)]$ define the stopping times $\hat{\alpha}_{y}, \hat{\beta}_{y}$ by
\begin{equation}\begin{aligned}
 \hat{\alpha}_{y} &:= \inf\{t \in [0,\infty)\colon X_t^{y} \leq \alpha(t) \},
 \\
 \hat{\beta}_{y} &:= \inf\{t \in [0,\infty)\colon X_t^{y} \geq \beta(t) \}. 
 \end{aligned}
\end{equation}
The quantities of interest in neurophysiological decision models are the \emph{first hitting time probabilities}: $\mathbb{P}[ \hat{\alpha}_{y} \leq \min(\tau,\hat{\beta}_{y}) ]$, where $\tau \in (0,\infty)$ and $y\in [\alpha(0),\beta(0)]$. These probabilities can be linked to the solution of a parabolic PDE. Indeed, assume $\alpha<\beta$ on $[0,\tau]$ for some $\tau\in (0,\infty)$, set $Q := \{(t,x)\in (0,\tau)\times\R \colon \alpha(\tau-t) < x < \beta(\tau-t))\}$, and consider the following PDE:
\begin{equation}\label{eq:FPintro}
 \left\{
\begin{alignedat}{2}
\partial_{t} F(t,x) &= \tfrac{\res{\sigma^2}}{2}\partial^2_{x} F(t,x) + \mu(\tau-t,x)  \partial_{x} F(t,x) && \quad (t,x)\in Q_{\tau},\\
F(t,\alpha(\tau-t))&=1, \quad F(t,\beta(\tau-t))=0 &&\quad t \in (0,\tau),\\
F(0, x)&=0&&\quad   x \in (\alpha(\tau), \beta(\tau)).
\end{alignedat}
\right.
\end{equation} 
Under some additional regularity assumptions on $\alpha$, $\beta$, and $\mu$ it can be shown that a solution to~\eqref{eq:FPintro} exists and satisfies 
\begin{equation}\label{eq:relFP_fhtp}\mathbb{P}[ \hat{\alpha}_{y} \leq \min(\tau,\hat{\beta}_{y}) ]  = F(\tau,y),\quad \alpha(0)\leq y \leq \beta(0).
\end{equation}
(see~\cite[Appendix A]{Voss2008} for the case that $\alpha$ and $\beta$ are constant and $\mu$ does not depend on time or~\cite[Chapter 7]{Oksendal1998} for general Fokker--Planck equations, also known \res{in this setting as a backward Kolmogorov equation}).\par 
In~\cite{Voss2008}, a Crank--Nicolson method is used to approximate solutions to~\eqref{eq:FPintro} in the case that $\alpha$, $\beta$, and $\mu$ are constant. One advantage of this setting is that one only needs to solve a single PDE of type~\eqref{eq:FPintro} in order to obtain
the first hitting time probabilities $\mathbb{P}[ \hat{\alpha}_{y} \leq \min(t,\hat{\beta}_{y}) ]$ for all $t\in [0,\tau],\, y\in [\alpha(0),\beta(0)]$. However, due to the fact that  $F$ is discontinuous at $(t,x)=(0,\alpha(\tau))$, no proof of convergence of the Crank--Nicolson for decreasing step-sizes seems available. At best, reduced rates are to be expected. Moreover, various authors have argued that time-dependent boundaries $\alpha$ and $\beta$ and space-time-dependent drift $\mu$ provide a more realistic model for decision processes, for an overview see~\cite{Hawkins2015,Shadlen2013}.

In this article we \emph{extend}~\cite{Voss2008} to include diffusion models with time-dependent boundaries and non-constant drift. We \emph{improve} the efficiency of the numerical simulation by not approximating the solution $F$ to~\eqref{eq:FPintro} directly, instead, we approximate the solution to a parabolic PDE on a rectangular domain with homogeneous initial and boundary conditions constructed such that its difference with $F$ (transformed to the same rectangular domain)  is a function for which a rapidly converging series expansion is known. 

More specifically, in Section~\ref{sec:FP} we demonstrate that if $\alpha,\beta$ are once continuously differentiable, then~\eqref{eq:FPintro} can be transformed into a parabolic PDE on a rectangular domain with a space-time-dependent drift. Next, in Section~\ref{sec:reganalysis} we demonstrate that by subtracting a known, discontinuous function, we obtain a parabolic PDE with homogeneous boundary conditions, see~\eqref{R3} below. 
We analyze the regularity of the solution $e$ to this equation and verify that it is indeed smoother than $F$, see Corollary~\ref{corol1} and Theorem~\ref{smoothness}.

In Section~\ref{SMRM} we apply a minimal residual method~\cite{11,249.99,249.992} to approximate the solution $e$ to~\eqref{R3}. 
This method is known to give quasi-best approximations from the selected trial space in the norm on a natural solution space being the intersection of two Bochner spaces.
Taking as trial space the space of continuous piecewise bilinears with respect to a uniform partition of the space-time cylinder into rectangles with mesh width $h$, in Theorem~\ref{thm:optimal convergence} the optimal error bound of order $h$ is shown for the solution~$e$ to~\eqref{R3}. 

In Section~\ref{sec:parameter dependence} we consider the situation that $\mu$, $\alpha$, and $\beta$ can be parametrized analytically and verify that in this case the corresponding solution $e$ to ~\eqref{R3} (transformed onto the unit square) depends analytically on these parameters as well as on the final time $\tau$, see Theorem~\ref{thm:analytic_in_para}. This justifies the use of a sparse tensor-product interpolation~\cite{239.18} to determine the solution $e$ to ~\eqref{R3} efficiently for multiple end-time and parameter values. Finally, in Section~\ref{sec:examples} we provide numerical simulations for three different decision models taken from the neurophysiological literature. \par 
\res{In our parallel publication~\cite{BCGS21} mentioned above, we provide further numerical experiments and code. There,} we apply the Crank--Nicolson method \res{(without giving any error analysis)} to approximate the solution $e$ to~\eqref{R3}. 
In the examples we consider it appears that the Crank--Nicolson method leads to similar convergence as the minimal residual method. Although \res{we only provide a rigorous error analysis for the minimal residual method}, Crank--Nicolson may be preferred in practice as it is easier to implement. We refer to~\cite{BCGS21} for further details.


\subsection{Notation}\label{ssec:1.2}
In this work, by $C \lesssim D$ we mean that $C$ can be bounded by a multiple of $D$, independently of parameters which $C$ and $D$ may depend on.
Obviously, $C \gtrsim D$ is defined as $D \lesssim C$, and $C\eqsim D$ as $C\lesssim D$ and $C \gtrsim D$.

For normed linear spaces $E$ and $F$, by $\cL(E,F)$ we denote the normed linear space of bounded linear mappings $E \rightarrow F$,
and by $\Lis(E,F)$ its subset of boundedly invertible linear mappings $E \rightarrow F$.

\section{Transforming the Fokker--Planck equation to a rectangular space-time domain} \label{sec:FP}
In this section we demonstrate that~\eqref{eq:FPintro} can be transformed into a PDE 
on a rectangular space-time domain, see~\eqref{R1} below. The PDE in~\eqref{R1} below forms the starting point for the remainder of this article, which is why 
we use tildes in~\eqref{R0} below to distinguish the variables and coefficients of the non-transformed equation from those in~\eqref{R1}. Indeed, let $\widetilde{T} \in (0,\infty]$, assume $a,b \in C^{1}([0,\widetilde{T}))$ satisfy $a(\tilde{t})< b(\tilde{t})$ for all $\tilde{t}\in [0,\widetilde{T})$,
set $\widetilde Q := \{ (\tilde{t},\tilde{x})\in (0,\widetilde{T})\times \R \colon a(\tilde{t}) < \tilde{x} < b(\tilde{t})\}$, let $\tilde{v} \in L_\infty(\widetilde Q)$, and consider the following parabolic initial- and boundary value problem:
\begin{equation}\label{R0}
\left\{
\begin{alignedat}{2}
\partial_{\tilde{t}} \tilde{u}(\tilde{t}, \tilde{x})&= \partial^2_{\tilde{x}} \tilde{u}(\tilde{t}, \tilde{x}) +\tilde{v}(\tilde{t}, \tilde{x}) \partial_{\tilde{x}} \tilde{u}(\tilde{t}, \tilde{x})&& \quad(\tilde{t}, \tilde{x}) \in \widetilde Q,\\
\tilde{u}(\tilde{t},a(\tilde{t}))&=1, \quad \tilde{u}(\tilde{t},b(\tilde{t}))=0 &&\quad \tilde{t} \in (0,\widetilde{T}),\\
\tilde{u}(0, \tilde{x})&=0&&\quad   \tilde{x} \in (a(0), b(0)).
\end{alignedat}
\right.
\end{equation}
Note that this is~\eqref{eq:FPintro} with $\tilde{u}(\tilde{t},\tilde{x}) = F(\frac{2\tilde{t}}{\res{\res{\sigma^\revv{2}}}},\tilde{x})$, $\widetilde{T}=\frac{\res{\sigma^2} \tau}{2}$, $a(\tilde{t})=\alpha(\frac{2}{\res{\sigma^2}}(\widetilde{T}-\tilde{t}))$, $b(\tilde{t})=\beta(\frac{2}{\res{\sigma^2}}(\widetilde{T}-\tilde{t}))$, $\tilde{v}(\tilde{t},\tilde{x})=\revv{\frac{2}{\sigma^2}}\mu(\frac{2}{\res{\sigma^2}}(\widetilde{T}-\tilde{t}),\tilde x )$.
%
%

Now, set $T:=\int_{0}^{\tilde{T}} |b(\tilde{s})-a(\tilde{s})|^{-2} \,d\tilde{s}$ \res{(where possibly $T=\infty$)} and define $\theta \colon [0,T) \rightarrow [0,\tilde{T})$ by 
 $\theta(t) = \sup\left\{ \tilde{r}\in [0,\widetilde{T}) \colon \int_{0}^{\tilde{r}} |b(\tilde{s})-a(\tilde{s})|^{-2} \,d\tilde{s} \leq t \right\}$,
then $\theta$ is a bijection and $\theta^{-1}(\tilde{t}) = \int_{0}^{\tilde{t}} |b(\tilde{s})-a(\tilde{s})|^{-2} \,d\tilde{s}$. In particular, from $t=\theta^{-1}(\theta(t))$ we obtain that $\theta$ satisfies the following ODE
\begin{equation}\label{eq:ode_time}
 \theta'(t) = \big(b(\theta(t)) - a(\theta(t))\big)^2,\quad \theta(0)=0.
\end{equation}

With
$$
\Omega:=(0,1),
$$
and $\xi\colon [0,\widetilde{T})\times \overline{\Omega} \rightarrow \R$ defined by
$$
\xi(\tilde{t},x) := (1-x) a(\tilde{t})+x b(\tilde{t}),
$$
we have that
$$
[0,T) \times \Omega \rightarrow \widetilde Q\colon (t,x) \mapsto (\theta(t),\xi(\theta(t),x))
$$
is a bijection with inverse
$$
(\tilde{t},\tilde{x}) \mapsto 
\Big(
    \theta^{-1}(\tilde{t}), 
    \frac{\tilde{x}-a(\tilde{t})}
    {b(\tilde{t})-a(\tilde{t})}
\Big).
$$

Defining $u,v\colon [0,T)\times \overline{\Omega} \rightarrow \R$ by
\begin{align*}
u(t,x) &:= \tilde{u}(\theta(t),\xi(\theta(t),x)),\\
v(t,x) &:= (b(\theta(t))-a(\theta(t))) \big[\tilde{v}(\theta(t),\xi(\theta(t),x))+(1-x)a'(\theta(t)))+x b'(\theta(t))\big],
\end{align*}
we have $u(t,0)=1$, $u(t,1)=0$ ($t \in (0,T)$), and $u(0,x)=0$ ($x \in \Omega$). Moreover, for $(t,x) \in (0,T)\times \Omega$, one has
\begin{align*}
 \partial_{x} u(t,x) &= (b(\theta(t))-a(\theta(t))) \partial_{\tilde{x}} \tilde{u}(\theta(t),\xi(\theta(t),x)),\\
 \partial^2_{x} u(t,x) &= (b(\theta(t))-a(\theta(t)))^2 \partial^2_{\tilde{x}} \tilde{u}(\theta(t),\xi(\theta(t),x)),
\end{align*}
and 
\begin{align*}
 \partial_t u(t,x) &=
 \theta'(t)\big\{\partial_{\tilde{t}} \tilde{u}(\theta(t), \xi(\theta(t),x))+\partial_{\tilde{t}} \xi(\theta(t),x)\partial_{\tilde{x}} \tilde{u}(\theta(t),\xi(\theta(t),x))\big\}
\\ & = (b(\theta(t))-a(\theta(t)))^2 \Big\{\partial_{\tilde{t}} \tilde{u}(\theta(t), \xi(\theta(t),x))+\\
&\hspace*{6em} \big[(1-x)a'(\theta(t))+x b'(\theta(t))\big] \partial_{\tilde{x}} \tilde{u}(\theta(t),\xi(\theta(t),x))\Big\}\\
& = (b(\theta(t))-a(\theta(t)))^2 \Big\{\partial^2_{\tilde{x}} \tilde{u}(\theta(t), \xi(\theta(t),x))+\\
&\big[\tilde{v}(\theta(t), \xi(\theta(t),x))+(1-x)a'(\theta(t))+x b'(\theta(t))\big] \partial_{\tilde{x}} \tilde{u}(\theta(t),\xi(\theta(t),x))\Big\}\\
& =\partial^2_{x} u(t,x)+ v(t,x) \partial_{x} u(t,x).
\end{align*} 
In other words, with
$$
I:=(0,T),
$$
\eqref{R0} is equivalent to finding $u=u(v)$ that solves
\begin{equation}\label{R1}
\left\{
\begin{alignedat}{2}
\partial_{t} u(t, x)&= \partial^2_{x} u(t, x) +v(t, x) \partial_{x} u(t, x)&& \quad(t, x) \in I \times \Omega,\\
u(t,0)&=1, \quad u(t,1)=0 &&\quad t \in I,\\
u(0, x)&=0&&\quad   x \in \Omega.
\end{alignedat}
\right.
\end{equation}
For its numerical solution, we always consider system \eqref{R1} for $T <\infty$.


\begin{example}\label{ex:linear boundaries}
Bowman, Kording, and Gottfried~\cite{Bowman2012} suggested 
collapsing boundaries, i.e., in~\eqref{eq:FPintro} they 
take $\alpha(t):= \frac{\beta_0 t}{2T_0}$ and $\beta(t) := \beta_0(1-\frac{t}{2T_0})$ for 
some fixed parameters $\beta_0, T_0 \in (0,\infty)$. Translating 
this to the setting of~\eqref{R0}, this leads to $a(\tilde{t}):=\frac{ \beta_0 (\widetilde{T}-\tilde{t})}{\res{\sigma^2}  T_0}$ and $b(\tilde{t}):= \beta_0(1-\frac{\widetilde{T}-\tilde{t}}{\res{\sigma^2} T_0})$ (note that it only makes sense to consider $\widetilde{T}\in (0,\frac{\res{\sigma^2} T_0}{2} )$ in this setting). 
Note that it is easier to first determine $\theta^{-1}(\tilde{t}) = \int_{0}^{\tilde{t}} |b(\tilde{s}) - a(\tilde{s})|^{-2} \,d\tilde{s}$ and then determine $T=\theta^{-1}(\tilde{T})$ and $\theta = (\theta^{-1})^{-1}$. Indeed, $\theta^{-1}(\tilde{t}) = \frac{\res{\sigma^4} T_0^2 \tilde{t}}{\beta_0^2(\res{\sigma^2} T_0 -2\widetilde{T})(\res{\sigma^2} T_0 -2\widetilde{T} +2\tilde{t})}$ and thus 
$T 
= \frac{\res{\sigma^2} T_0 \widetilde{T}}{\beta_0^2(\res{\sigma^2} T_0 - 2\widetilde{T})}$
and  
$$
 \theta(t) 
 = \frac{ \beta_0^2(\res{\sigma^2} T_0- 2\widetilde{T})^2 t }
 { \res{\sigma^4} T_0^2 - 2\beta_0^2(\res{\sigma^2} T_0- 2 \widetilde{T})t },\quad t\in [0,T).
$$
By observing that $(1-x)a'(\theta(t))+x b'(\theta(t))=\frac{(2x-1) \beta_0}{\res{\sigma^2} T_0}$, and
 $b(\theta(t))-a(\theta(t))$ $= \frac{b_0(1-2\res{\sigma^2} T_0 \widetilde{T})}{\res{\sigma^4} T_0^2-2 \beta_0^2(\res{\sigma^2} T_0-2 \widetilde{T}) t}$, one obtains $v$ in terms of $\tilde v$.
\end{example}


\section{\res{Regularity} of the Fokker--Planck equation}\label{sec:reganalysis} 
Let $u(v)$ denote the solution to~\eqref{R1} for some given drift function $v$. Due to the discontinuity between boundary and initial data, it is clear that $u(v)$ is discontinuous at the corner $(t,x)=(0,0)$. This reduces the rate of convergence of standard numerical methods and makes it difficult to provide a theoretical bound on the convergence rate. However, for \emph{constant drift} $v$, a rapidly converging series expansion of $u(v)$ is known (\cite{168.84}), which allows to efficiently approximate $u(v)$ within any given positive tolerance. 
Knowing this, our approach to approximate $u(v)$ for \emph{variable} $v \in C(\overline{I \times \Omega})$ is to \res{\emph{approximate the difference}
$$
e=e(v)=u(v)-u(v_{0}), \text{ where } v_{0}:=v(0,0).
$$}
This function $e(v)$ solves
\be \label{R3}
\left\{
\begin{alignedat}{2}
\partial_t e(t, x)&= \partial^2_x e(t, x) \!+\! v(t, x) \partial_x e(t, x)\!+\! (v(t, x)-v_{0}) \partial_x u(v_{0})&& \quad(t, x) \in I\times \Omega,\\
e(t,0)&=0, \quad e(t,1)=0 &&\quad t \in I,\\
e(0, x)&=0&&\quad   x \in \Omega,
\end{alignedat}
\right.
\ee
which we solve approximately with a numerical method.
To derive a priori bounds for the approximation error, we analyze the smoothness of $e(v)$, see Section~\ref{ssec:regularity_e}. In particular, under additional smoothness conditions on $v$, and using that $(v-v_{0})(0,0)=0$, we show that 
\res{$$
e(v) \text{ \emph{is more smooth than} } u(v_{0}), \text{ \emph{and thus than} }u(v),
$$}%
which shows the benefit of applying the numerical method to \eqref{R3} instead of directly to \eqref{R1}.

It turns out that for any $v$ the smoothness of $u(v)$ is determined by that of the solution $u_H$ of the heat equation on $(0,\infty) \times \R$ that is $0$ at $t=0$ and $1$ at $x=0$. Its smoothness is the topic of the next subsection.

\subsection{The heat kernel} \label{Sheat}
The function
$$
H(t,x):=\frac{1}{2\sqrt{\pi t}} e^{-\frac{x^2}{4t}}
$$
is the heat kernel. It satisfies 
\begin{equation*}
 \begin{aligned}
\partial_t H(t,x)& =\partial^2_x H(t,x) &&(t, x) \in (0,\infty) \times \R,\\
\lim_{t \downarrow 0} \int_\R H(t,x) \phi(x) \,d x&=\phi(0) && 
\revision{\text{for all } \phi \in \mathcal{D}(\R),}
\end{aligned}
\end{equation*}
\revision{the latter being the space of \emph{test functions}.}

Following \cite[Ex.~2.14]{45.491} and \cite{75.215}, for $(t, x) \in (0,\infty) \times \R$ we define
$$
u_H(t,x):=2 \int_x^\infty H(t,y)\,dy={\textstyle \frac{2}{\sqrt{\pi}}} \int_{\frac{x}{2\sqrt{t}}}^\infty e^{-s^2}\,ds=\Erfc({\textstyle \frac{x}{2\sqrt{t}}}).
$$
Knowing that $\int_{0}^\infty {\textstyle \frac{1}{\sqrt{\pi t}}} e^{-\frac{y^2}{4t}} \,dy=1$, and 
${\displaystyle \lim_{t \downarrow 0}} \int_{x}^\infty {\textstyle \frac{1}{\sqrt{\pi t}}} e^{-\frac{y^2}{4t}} \,dy=0$ for $x>0$, we have 
$$
\left\{
\begin{alignedat}{2}
\partial_t u_H(t,x)& =\partial^2_x u_H(t,x) && \quad(t, x) \in (0,\infty) \times \R,\\
u_H(t,0)& =1 && \quad t>0,\\
u_H(0,x):=\lim_{t \downarrow 0} u_H(t,x)&=0&& \quad x>0.
\end{alignedat}
\right.
$$



The following lemma turns out to be handy to analyze the smoothness of $u_H$ restricted to $I \times \Omega$.

\begin{lemma} \label{lemmie3}
For $p>0$, $\alpha,\beta \in \R$, it holds that $\int_0^T \int_0^1 |t^\alpha x^\beta e^{-\frac{x^2}{4t}}|^p \,dx\,dt<\infty$ if and only if $p \beta>-1$ and $p(2\alpha +\beta) >-3$.
\end{lemma}


\begin{proof} 
The mapping
$$
\Phi \colon \big\{(\lambda,x) \in (0,\infty)\times (0,1)\colon x<2 \sqrt{\lambda T}\big\} \rightarrow (0,T)\times (0,1)\colon (\lambda,x) \mapsto \big(\frac{x^2}{4\lambda},x\big)
$$ 
is a diffeomorphism, and $|\Omega\Phi(\lambda,x)|=\frac{x^2}{4 \lambda^2}$. One obtains
\begin{align*}
\int_0^T\int_0^1 |t^\alpha x^\beta e^{-\frac{x^2}{4t}}|^p \,dx\,dt=
\int_0^\infty \int_0^{\min(1,2 \sqrt{\lambda T})} \big(\frac{x^2}{4 \lambda}\big)^{\alpha p} x^{\beta p} e^{-p \lambda}\frac{x^2}{4 \lambda^2} \,dx\,d\lambda\\
= 4^{-\alpha p -1} \int_0^\infty \lambda^{-\alpha p -2} e^{-p \lambda} \int_0^{\min(1,2 \sqrt{\lambda T})}  x^{2 \alpha p+\beta p+2}\,dx \,d \lambda.
\end{align*}
The integral over $x$ is finite if and only if $p(2\alpha +\beta) >-3$, and if so, the expression is equal to
$$
{\textstyle \frac{4^{-\alpha p -1}}{2 \alpha p+\beta p+3}} \Big[ (2\sqrt{T})^{2\alpha p +\beta p+3}
\int_0^\frac{1}{4T} \lambda^{(\beta p-1)/2} e^{-p \lambda} d \lambda+ \int_\frac{1}{4T}^\infty \lambda^{-\alpha p -2} e^{-p \lambda} \,d\lambda\Big]
$$ with the first integral being finite if and only if $p \beta>-1$.
\end{proof}

Following~\cite{314.91}, we analyze the regularity \res{of the solutions $u(v)$ and $e(v)$ of the parabolic problems ~\eqref{R1} and~\eqref{R3}, respectively,}  in (intersections of) Bochner spaces.
In particular, the space $L_2(I;H^1(\Omega)) \cap H^1(I;H^{-1}(\Omega))$ plays an important role in this and following sections. For the precise definition of this space and some properties we refer to~\cite[Chapter 25]{314.91}.
With $H^1_{0,\{0\}}(I)$ denoting the closure in $H^1(I)$ of the functions in $C^\infty(I) \cap H^1(I)$ that vanish at $0$,  
we have the following result concerning the smoothness of $u_H$  restricted to $I \times \Omega$.

\begin{corollary} \label{corol1} $u_H \in L_2(I;H^1(\Omega)) \cap H^1_{0,\{0\}}(I;H^{-1}(\Omega))$, but $u_H \not\in H^1_{0,\{0\}}(I;L_2(\Omega))$ and $u_H \not\in L_2(I;H^2(\Omega))$.
Furthermore, $t \partial_t \partial_x u_H, x \partial^2_x u_H, t \partial^2_x u_H \in L_2(I \times \Omega)$, and 
$x \partial_t \partial_x u_H \in L_2(I;H^{-1}(\Omega))$.
\end{corollary}

\begin{proof} By applications of Lemma~\ref{lemmie3}, we infer that $\partial_x u_H=-2H \in L_2(I \times \Omega)$, and that
$\partial_t u_H(t,x)=\frac{1}{2\sqrt{\pi}} x t^{-\frac32} e^{-\frac{x^2}{4t}} \not\in L_2(I \times \Omega)$.
This yields $u_H \in L_2(I;H^1(\Omega))$ and $u_H \not \in H_{0,\{0\}}^1(I;L_2(\Omega))$.

If $\partial_x F=f$, then $f \in L_2(I;H^{-1}(\Omega))$ if and only if $F \in L_2(I \times \Omega)$. We have
$\int_{-\infty}^x \partial_t u_H(t,y)\,dy=-\frac{t^{-\frac12}}{\sqrt{\pi}}e^{-\frac{x^2}{4t}} \in L_2(I \times \Omega)$, so indeed $u_H \in H^1_{0,\{0\}}(I;H^{-1}(\Omega))$.

It holds
$$
\partial_x^2 u_H(t,x)=-2\partial_x H(t,x)=\frac{1}{2 \sqrt{\pi}} x t^{-\frac32} e^{-\frac{x^2}{4t}} \not\in L_2(I \times \Omega),
$$
or $u_H \not\in L_2(I;H^2(\Omega))$, but $ x \partial^2_x u_H, t \partial^2_x u_H \in L_2(I \times \Omega)$.

We have $\partial_x\partial_t u_H=(t^{-\frac32}-\frac12 x^2 t^{-\frac52}) \frac{e^{-\frac{x^2}{4t}}}{2 \sqrt{\pi}}$, so
$t \partial_x\partial_t u_H \in L_2(I \times \Omega)$. 
Proving that $x \partial_x\partial_t u_H \in L_2(I;H^{-1}(\Omega))$ amounts to proving  
$x t^{-\frac32}e^{-\frac{x^2}{4t}},\,t^{-\frac52}x^3e^{-\frac{x^2}{4t}} \in L_2(I;H^{-1}(\Omega))$, i.e., proving that
$t^{-\frac32} \int_{-\infty}^x y e^{-\frac{y^2}{4t}}dy,\,t^{-\frac52}\int_{-\infty}^x y^3e^{-\frac{y^2}{4t}} dy \in L_2(I \times \Omega)$.
The first function equals $-2t^{-\frac12}e^{-\frac{x^2}{4t}}$, which is in $L_2(I \times \Omega)$,
and the second function equals $-8t^{-\frac12} e^{-\frac{x^2}{4t}} -2t^{-\frac32} x^2  e^{-\frac{x^2}{4t}}$, which is also in $L_2(I \times \Omega)$.
\end{proof}

Finally in this subsection, notice that from $\res {\partial_t} u_H(t,x)=\frac{1}{2\sqrt{\pi}} x t^{-\frac32} e^{-\frac{x^2}{4t}}$, it follows that for any $x>0$ and $k \in \N_0$,
\be \label{e3}
\lim_{t \downarrow 0} \res {\partial^k_t} u_H(t,x)=0.
\ee

\subsection{Regularity of the parabolic problem with homogeneous initial and boundary conditions} \label{sec:reg}
Knowing that $e(v)$ is the solution of the parabolic problem \eqref{R3} that has homogeneous initial and boundary conditions, we study the regularity of such a problem.

Given functions $v\in L_{\infty}(I\times \Omega)$ and $f\in L_2(I;H^{-1}(\Omega))$, let $w$ solve
\be \label{Rhom}
\left\{
\begin{alignedat}{2}
\partial_t w(t, x)&= \partial^2_x w(t, x) +v(t, x) \partial_x w(t, x)+f(t,x)&& \quad(t, x) \in I\times \Omega,\\
w(t,0)&=0, \quad w(t,1)=0 &&\quad t \in I,\\
w(0, x)&=0&&\quad   x \in \Omega,
\end{alignedat}
\right.
\ee
where the spatial differential operators at the right-hand side should be interpreted in a weak sense, i.e.,
$((\partial_x^2 +v\partial_x)\eta)(\zeta):=\int_D -\partial_x \eta \partial_x \zeta+ v\partial_x \eta \,\zeta\,dx$.
It is well-known that
\be \label{defL}
L(v):=w \mapsto f \in  \Lis(L_2(I;H^1_0(\Omega)) \cap H^1_{0,\{0\}}(I;H^{-1}(\Omega)),L_2(I;H^{-1}(\Omega)))
\ee
(see, e.g., \cite[Thm.~26.1]{314.91}).
Under additional smoothness conditions on the right-hand side $f$ beyond being in $L_2(I;H^{-1}(\Omega))$, additional  smoothness of the solution $w$ can be demonstrated:

\begin{proposition} \label{reg} a) If $v \in W^1_\infty(I \times \Omega)$, then
\begin{align*}
L(v)^{-1} \in \cL\Big(&
 L_2(I;H^1(\Omega)) \cap  H^1(I;H^{-1}(\Omega)),\\
&H^1_{0,\{0\}}(I;H^1_0(\Omega))\cap H^2(I;H^{-1}(\Omega)) \cap L_2(I;H^3(\Omega))
\Big).
\end{align*}

b) If $v \in L_\infty(I\times \Omega)$, then 
$$
L(v)^{-1}  \in \cL\Big(L_2(I \times \Omega),L_2(I;H^2(\Omega)) \cap H^1_{0,\{0\}}(I;L_2(\Omega))\Big),
$$
\end{proposition}

\begin{proof} a) If $f \in L_2(I;H^1(\Omega)) \cap  H^1(I;H^{-1}(\Omega))$, then also
$f \in H^1(I;H^{-1}(\Omega))$, and $f(0,\cdot) \in L_2(\Omega)$ with $\|f(0,\cdot)\|_{L_2(\Omega)} \lesssim \|f\|_{L_2(I;H^1(\Omega))}$ $+$ $\|f\|_{H^1(I;H^{-1}(\Omega))}$ (see, e.g., ~\cite[Thm.~25.5]{314.91}). As shown in \cite[Thm.~27.2 and its proof]{314.91}, from the last two properties of $f$, and $v \in W_\infty^1(I;L_\infty(\Omega))$, one has $w=L(v)^{-1} f \in H^1_{0,\{0\}}(I;H^1_0(\Omega))\cap H^2(I;H^{-1}(\Omega))$ with 
\begin{equation*}\|w \|_{H^1_{0,\{0\}}(I;H^1_0(\Omega))\cap H^2(I;H^{-1}(\Omega))} \lesssim \|f\|_{H^1(I;H^{-1}(\Omega))}+\|f(0,\cdot)\|_{L_2(\Omega)}.
\end{equation*}

To show the spatial regularity, i.e., $w \in L_2(I;H^3(\Omega))$, given a constant $\lambda$, we define $w_\lambda(t,\cdot)=w(t,\cdot) e^{-\lambda t}$, $f_\lambda(t,\cdot)=f(t,\cdot) e^{-\lambda t}$. One infers that
\begin{equation}\label{eq:FP_collapsed}
(-\partial_x^2 -v \partial_x+\lambda)w_\lambda=\underbrace{f_\lambda-\partial_t w_\lambda}_{\res{g_\lambda:=}} \,\,\,\text{ on } I \times \Omega,\quad w_\lambda(\cdot,0)=0=w_\lambda(\cdot,1) \,\,\, \text{ on }  I,
\end{equation}
where, as before, the spatial differential operators should be interpreted in a weak sense.
Using that 
\begin{equation*} 
\left|\int_I \int_D v (\partial_x w_\lambda )\,w_\lambda\,dx\,dt \right|\leq \|v\|_{L_\infty(I \times \Omega)} \|\partial_x w_\lambda\|_{L_2(I\times \Omega)} \|w_\lambda\|_{L_2(I \times \Omega)}
\end{equation*}
and Young's inequality, one infers that for $\lambda > \frac{1}{4}\|v\|^2_{L_\infty(I \times \Omega)}$ the bilinear form defined by the left-hand side of~\eqref{eq:FP_collapsed} is bounded and \emph{coercive} on $L_2(I;H^1_0(\Omega)) \times L_2(I;H^1_0(\Omega))$.
Thus for $\lambda > \frac{1}{4}\|v\|^2_{L_\infty(I \times \Omega)}$ we have
$$
A(v,\lambda)\res{:=w_\lambda \mapsto g_\lambda} \in\Lis(L_2(I;H^1_0(\Omega)),L_2(I;H^{-1}(\Omega))).$$
Realizing that $\|\cdot\|_{H^{k+2}(\Omega)}^2=\res{\|\frac{\mathrm{d}^k}{\mathrm{d}x^k} \frac{\mathrm{d}^2}{\mathrm{d} x^2}\cdot\|_{L_2(\Omega)}^2}+\|\cdot\|_{H^{k+1}(\Omega)}^2$, an induction and tensor product argument shows $A(0,0)^{-1} \in \cL(L_2(I;H^k(\Omega)),L_2(I;H^{k+2}(\Omega)))$ for any $k \in \N_0$.
Writing
$$
A(v,\lambda)^{-1}-A(0,0)^{-1}=A(0,0)^{-1}(v \partial_x-\lambda \identity)A(v,\lambda)^{-1},
$$
and using that $v \partial_x \in \cL(L_2(I;H^1(\Omega)), L_2(I;L_2(\Omega))$ by $v \in L_\infty(I\times \Omega)$, one verifies that $
A(v,\lambda)^{-1} \in \cL(L_2(I \times \Omega),L_2(I;H^{2}(\Omega)))$. Repeating the argument, now using that 
$v \partial_x \in \cL(L_2(I;H^2(\Omega)), L_2(I;H^1(\Omega))$ by $v \in L_\infty(I;W_\infty^1( \Omega))$,
one has $
A(v,\lambda)^{-1} \in \cL(L_2(I;H^1(\Omega)),L_2(I;H^{3}(\Omega)))$. Knowing that $f_\lambda-\partial_t w_\lambda$ $\in$ $L_2(I;H^1(\Omega))$ with $\| f_\lambda-\partial_t w_\lambda\|_{L_2(I;H^1(\Omega))} \lesssim \|f\|_{L_2(I;H^1(\Omega))}+\|f\|_{H^1(I;H^{-1}(\Omega))}$, 
one infers that $w_\lambda$ and thus $w \in L_2(I;H^{3}(\Omega))$, and moreover $\| w \|_{L_2(I;H^{\res{3}}(\Omega))} \lesssim \|f\|_{L_2(I;H^1(\Omega))}+\|f\|_{H^1(I;H^{-1}(\Omega))}$.

b) Similar to Part a), it suffices to show that
$$
L(v,\lambda)^{-1}:=f_\lambda \mapsto w_\lambda \in  \cL\Big(L_2(I \times \Omega),L_2(I;H^2(\Omega)) \cap H^1_{0,\{0\}}(I;L_2(\Omega))\Big).
$$
Knowing that $L(v,\lambda)^{-1} \in \cL\big(L_2(I;H^{-1}(\Omega)), L_2(I;H^1_0(\Omega)) \cap H^1_{0,\{0\}}(I;H^{-1}(\Omega))\big)$, and $L(v,\lambda)-L(0,0) = -v \partial_x+\lambda \identity \in \cL\big(L_2(I;H^1_0(\Omega)),L_2(I\times \Omega)\big)$, the proof is completed by  $L(v,\lambda)^{-1}-L(0,0)^{-1}=L(0,0)^{-1}(L(0,0)-L(v,\lambda))L(v,\lambda)^{-1}$ and the
\emph{maximal regularity} result 
$$
L(0,0)^{-1} \in \cL\Big(L_2(I \times \Omega),L_2(I;H^2(\Omega)) \cap H^1_{0,\{0\}}(I;L_2(\Omega))\Big)
$$
from, e.g., \cite{64.57,64.53}. 
\end{proof}

\subsection{The regularity of \bm{$e(v)=u(v)-u(v_{0})$}}\label{ssec:regularity_e}
Recall that $u_H$ denotes the solution of the heat equation studied in Section~\ref{Sheat}, that $u(v)$ denotes the solution to~\eqref{R3} for given $v\in C(\overline{I\times \Omega})$,
and $v_0:=v(0,0)$.
Since $e(v)$ solves \eqref{R3}, i.e., $e(v)$ is the solution $w$ of \eqref{Rhom} for forcing function $f$ given by 
\begin{equation}\label{split}
\begin{aligned}
&(v-v_{0})\partial_x u(v_{0}) 
\\ & =
(v-v_{0})\partial_x (u(v_{0})-u(0))+
(v-v_{0})\partial_x (u(0)-u_H)+
(v-v_{0})\partial_x u_H,
\end{aligned}
\end{equation}
in view of the regularity results proven in Proposition~\ref{reg}, we establish smoothness of $e(v)$ by demonstrating smoothness of each of the three terms at the right-hand side of \eqref{split}.

\begin{lemma} \label{corol2} It holds that
$$
u(0)-u_H \in H^1_{0,\{0\}}(I;H^1_0(\Omega))\cap H^2(I;H^{-1}(\Omega)) \cap L_2(I;H^3(\Omega)).
$$
\end{lemma}

\begin{proof}
The function $w(t,x):=u(0)(t,x)-(u_H(t,x)-xu_H(t,1))$ satisfies the homogeneous initial and boundary conditions from \eqref{Rhom}, and $\partial_t w(t,x)=\partial_x^2 w(t,x) +x \partial_t u_H(t,1)$. By \eqref{e3} we have 
$(t,x) \mapsto x \partial_t u_H(t,1) \in  L_2(I;H^1(\Omega)) \cap  H^1(I;H^{-1}(\Omega))$, so that Proposition~\ref{reg}a) for $v=0$  and $f(t,x)= x\partial_t u_H(t,1)$ shows that
$$
w \in H^1_{0,\{0\}}(I;H^1_0(\Omega))\cap H^2(I;H^{-1}(\Omega)) \cap L_2(I;H^3(\Omega)).
$$
Because, again by \eqref{e3}, $(t,x) \mapsto xu_H(t,1)$ is in the same space, the proof is completed.
\end{proof}

\begin{lemma} \label{corol3} For \emph{any} $v_0\in\R$, $u(v_0) -u(0) \in L_2(I;H^2(\Omega)) \cap H^1_{0,\{0\}}(I;L_2(\Omega))$.
\end{lemma}

\begin{proof}
The function $w:=u(v_0)-u(0)$ satisfies the homogeneous initial- and boundary conditions from \eqref{Rhom}, and $\partial_t w(t,x)=\partial_x^2 w(t,x) +v_0 \partial_x w-v_0 \partial_x u(0)$. From $\partial_x u(0) \in L_2(I \times \Omega)$ by Corollary~\ref{corol1} and Lemma~\ref{corol2}, an application of Proposition~\ref{reg}b) for $v = v_0$ and $f=-v_0 \partial_x u(0)$ completes the proof.
\end{proof}

\begin{lemma} \label{corol4} If $v \in W^1_\infty(I \times \Omega) \cap L_\infty(I;W^2_\infty(\Omega))$, then
$$
(v-v_{0}) \partial_x u_H \in L_2(I;H^1(\Omega)) \cap H^1(I;H^{-1}(\Omega)).
$$
\end{lemma}

\begin{proof}
Abbreviate $g:=(v-v_{0}) \partial_x u_H$.
Throughout the proof, we use the estimates for $u_H$ proven in Corollary~\ref{corol1}.

We start with proving $\partial_t g =(\partial_t v) \partial_x u_H +(v-v_{0} ) \partial_t \partial_x u_H \in L_2(I;H^{-1}(\Omega))$.
Using $v \in W_\infty^1(I;L_\infty(\Omega))$ and $\partial_x u_H  \in L_2(I \times \Omega)$, the first term is even in $ L_2(I \times \Omega)$.
Writing the second term as
$$
(v(t,x)-v_{0} )\partial_t\partial_x u_H (t,x)={\textstyle \frac{v(t,x)-v(0,x)}{t}}t \partial_t\partial_x u_H (t,x)+{\textstyle \frac{v(0,x)-v_{0} }{x}}x \partial_t\partial_x u_H (t,x),
$$
 from $t \partial_t\partial_x u_H  \in L_2(I \times \Omega)$ and ${\textstyle \frac{v(t,x)-v(0,x)}{t}} \in L_\infty(I\times \Omega)$ by $v \in W^1_\infty(I;L_\infty(\Omega))$, we have 
${\textstyle \frac{v(t,x)-v(0,x)}{t}}t \partial_t\partial_x u_H (t,x) \in L_2(I\times \Omega)$.
Similarly, from $x \partial_t\partial_x u_H (t,x) \in L_2(I;H^{-1}(\Omega))$ and
${\textstyle \frac{v(0,x)-v_{0} }{x}} \in L_\infty(I;W^1_\infty(\Omega))$ by $v \in L_\infty(I;W^2_\infty(\Omega))$, we have 
${\textstyle \frac{v(0,x)-v_{0} }{x}}x \partial_t\partial_x u_H (t,x)$ $\in$ $L_2(I;H^{-1}(\Omega))$, so that $\partial_t g \in L_2(I;H^{-1}(\Omega))$.

It remains to show that $g  \in L_2(I;H^1(\Omega))$. It is clear that $(v-v_{0} )\partial_x u_H  \in L_2(I \times \Omega)$ and $(\partial_x v) \partial_x u_H  \in L_2(I\times \Omega)$ by $v \in L_\infty(I;W_\infty^1(\Omega))$.
Writing 
$$
(v(t,x)-v_{0} )\partial^2_x u_H (t,x)={\textstyle \frac{v(t,x)-v(0,x)}{t}}t \partial^2_x u_H (t,x)+{\textstyle \frac{v(0,x)-v_{0} }{x}}x \partial^2_x u_H (t,x),
$$
from ${\textstyle \frac{v(t,x)-v(0,x)}{t}},\,{\textstyle \frac{v(0,x)-v_{0} }{x}}\in L_\infty(I \times \Omega)$ by $v \in W^1_\infty(I\times \Omega)$, and both $t \partial^2_x u_H (t,x)$ and $x \partial^2_x u_H (t,x) \in L_2(I\times \Omega)$, we obtain  $g  \in L_2(I;H^1(\Omega))$, and the proof is completed.
\end{proof}

By combining the results of the preceding three propositions with the regularity result proven in Proposition~\ref{reg} we obtain the following.
\begin{theorem} \label{smoothness} If $v \in W^1_\infty(I \times \Omega) \cap L_\infty(I;W^2_\infty(\Omega))$, then
$$
e(v)\in H^1_{0,\{0\}}(I;H^1_0(\Omega))\cap H^2(I;H^{-1}(\Omega)) \cap L_2(I;H^3(\Omega)).
$$
\end{theorem}

\begin{proof} We obtain $(v-v_{0})\partial_x(u(v_{0})-u_H)$  $\in$  $L_2(I;H^1(\Omega))$ $\cap$ $H^1(I;H^{-1}(\Omega))$ from Lemma~\ref{corol2} and \ref{corol3}, whereas Lemma~\ref{corol4} implies that $(v-v_{0})\partial_x u_H \in L_2(I;H^1(\Omega)) \cap H^1(I;H^{-1}(\Omega))$. 
We conclude that $$(v-v_{0})\partial_xu(v_{0}) \in L_2(I;H^1(\Omega)) \cap H^1(I;H^{-1}(\Omega)),$$ so that an application of Proposition~\ref{reg}a) completes the proof.
\end{proof}

Notice that as a consequence of Corollary~\ref{corol1}, Lemma~\ref{corol2} and \ref{corol3}, 
 $u(v_0)\notin H_{0,\{0\}}^1(I; L_2(\Omega))\cup L_2(I\res{;}H^2(\Omega))$. Comparing Corollary~\ref{corol1} with Theorem~\ref{smoothness}, we conclude that
 \res{$$
e(v)=u(v)-u(v_0) \text{ \emph{is indeed more smooth than} } u(v_{0}), \text{ \emph{and thus than} } u(v),
$$}%
confirming the claim we made at the beginning of Section~\ref{sec:reganalysis}.
\medskip


\section{Minimal residual method} \label{SMRM}
For solving \eqref{Rhom} (specifically for the forcing function $f$ as in \eqref{split}, \res{i.e., for solving $e(v)$}), we write it in variational form, i.e., we multiply it by test functions $z\colon I\times \Omega \rightarrow \R$ from a suitable collection, integrate it over $I \times \Omega$, and apply integration by parts with respect to $x$. We thus arrive at
\begin{align*}
&(B w)(z):=\\
&\int_I \int_D \partial_t w(t,x) z(t,x)+\partial_x w(t,x) \partial_x z(t,x)-v(t,x)\partial_x w(t,x)z(t,x)\,dx\,dt\\
&=
\int_I \int_D f(t,x)z(t,x)\,dx\,dt=:f(z)
\end{align*}
for all those test functions.
With
$$
X:=L_2(I;H^1_0(\Omega)) \cap H^1(I;H^{-1}(\Omega)), \quad Y:=L_2(I;H^1_0(\Omega)),
$$
it is known that $(B,\gamma_0) \in \Lis(X ,Y'\times L_2(\Omega))$, 
where $\gamma_0:= w \mapsto w(0,\cdot)$ denotes the initial trace operator, see, e.g., \cite[Chapter IV]{314.91} or~\cite{247.15}.

Already because $X \neq Y \times L_2(\Omega)$, the well-posed system \res{$(B,\gamma_0)w=(f,0)$} cannot be discretized by simple Galerkin discretizations.
Given a family $(X_h)_{h \in \Delta}$ of finite dimensional subspaces of $X$, as discrete approximations to $w$ one may consider the minimizers $\argmin_{\bar{w} \in X_h} \|B \bar{w}-f\|^{2}_{Y'} +\|\gamma_0 \bar{w}\|^{2}_{L_2(\Omega)}$. Since the dual norm $\|\cdot\|_{Y'}$ cannot be evaluated, this approach is not immediately feasible either. Therefore, for $(Y_h)_{h \in \Delta}$ being a second family of finite dimensional subspaces, now of $Y$, for $h \in \Delta$ as a discrete approximation from $X_h$ we consider
\be \label{MRM}
w_h:=\argmin_{\bar{w} \in X_h} \|B \bar{w}-f\|^2_{Y_h'} +\|\gamma_0\bar{w}\|^{2}_{L_2(\Omega)}.
\ee
This minimal residual approach has been studied for general parabolic PDEs in, e.g., \cite{11,249.99,249.992}, where $\Omega$ can be a $d$-dimensional spatial domain for arbitrary $d \geq 1$. 

For parabolic differential operators with a possibly asymmetric spatial part, in our setting caused by a non-zero drift function $v$,
in \cite[Thm.~3.1]{249.992} it has been shown that if $X_h \subset Y_h$ and
\be \label{infsup}
\varrho:=\inf_{h \in \Delta} \inf_{0 \neq \bar{w} \in X_h} \frac{\|\partial_t \bar{w}\|_{Y_h'}}{\|\partial_t \bar{w}\|_{Y'}}>0,
\ee
then
\be \label{quasi}
\|w-w_h\|_X \lesssim \min_{\bar{w} \in X_h} \|w-\bar{w}\|_X,
\ee
where the implied constant in~\eqref{quasi} depends only on $\varrho$ and an upper bound for $\|v\|_{L_\infty(I \times \Omega)}$,
i.e., $w_h$ is a \emph{quasi-best approximation} from $X_h$ with respect to the norm on $X$.

\begin{remark} \label{opmerking} This quasi-optimality result has been demonstrated under the condition that 
the spatial part of the parabolic differential operator is \emph{coercive} on $H^1_0(\Omega) \times H^1_0(\Omega)$ for a.e.~$t \in I$, i.e., 
$$
\int_D \eta' \eta'-v(t,\cdot) \eta' \eta\,dx \gtrsim \|\eta\|^2_{H^1(\Omega)} \quad (\eta \in H_0^1(\Omega)),
$$
which holds true when $\partial_x v \leq 0$ or  $\|v\|_{L_\infty(I\times \Omega)} \sup_{0 \neq \eta \in H^1_0(\Omega)} \frac{\|\eta\|_{L_2(\Omega)}}{\|\eta'\|_{L_2(\Omega)}}<1$, but which might be violated otherwise.

Although this coercivity condition might not be necessary, it can always be enforced  by considering $w_\lambda(t,\cdot):=w(t,\cdot) e^{-\lambda t}$, $f_\lambda(t,\cdot):=f(t,\cdot) e^{-\lambda t}$ instead of $w$ and $f$ with $\lambda$ sufficiently large, see also the proof of Proposition~\ref{reg}. 
By approximating $w_\lambda$ by the minimal residual method, and by multiplying the obtained approximation by $e^{\lambda t}$, an approximation for $w$ is obtained. Since qualitatively the transformations with $e^{\pm\lambda t}$ do not affect the smoothness of solution or right-hand side, for convenience in the following we pretend that coercivity holds true for  \eqref{Rhom}.
\end{remark}

As in~\cite{249.992,249.99}, we equip $Y_h$ in \eqref{MRM} with the energy norm
$$
\| z\|_Y^2 := (A_s z)(z) \quad(z\in Y_h), 
$$
where
\begin{align*}
&(A_s z)(\bar z) :=\\
& \int_I \int_D \partial_x z(t,x)\partial_x \bar z (t,x) - \frac{v(t,x)}{2} (\partial_x z(t,x) \bar z(t,x) + z(t,x) \partial_x \bar z(t,x)) \,d x\,dt
\end{align*}
denotes the symmetric part of the spatial differential operator.
Equipping $Y_h$ and  $X_h$ with bases $\Phi^h=\{\phi^h_i\}$ and $\Psi^h=\{\psi^h_j\}$, \res{respectively}, and denoting by $\bm{w}^h$ the representation of the minimizer $w_h$ with respect to~$\Psi_h$, $\bm{w}^h$ is found as the second component of the solution of
\begin{align}\label{eq:discrete system}
\left[\begin{array}{@{}cc@{}} \bm{A}_s^h & \bm{B}^h \\
{\bm{B}^h}^\top & \bm{C}^h
\end{array}\right]
\left[\begin{array}{@{}c@{}} \bm{\mu}^h \\ \bm{w}^h \end{array}\right]
=
\left[\begin{array}{@{}c@{}} \bm{f}^h \\ 0\end{array}\right],
\end{align}
where $(\bm{A}_s^h)_{i j}:= (A_s\phi^h_j)(\phi^h_i)$, $\bm{B}^h_{i j}:=(B \psi^h_j)(\phi^h_i)$, $\bm{C}^h_{i j}:= \int_D \psi^h_j(0,x) \psi^h_i(0,x) \, dx$, and
$\bm{f}^h_i:=f(\phi^{\res{h}}_i)$.
The operator $A_s$ can be replaced by any other spectrally equivalent operator on $Y_h$ 
without compromising the quasi-optimality result \eqref{quasi}. We refer to~\cite{249.992,249.99} for details.

Let $P_1$ be the set of polynomials of degree one. 
Taking for $n:=1/h \in \N$,
\begin{align}
V_{x,h}&:=\big\{\eta\in H^1_0(\Omega)\colon \eta|_{((i-1)h,ih)} \in P_1 \text{ for } i=1,\dots,n \big\},\notag
\\
V_{t,h}&:=\big\{\zeta\in H^1(I)\colon  \zeta|_{((i-1)hT,ihT)} \in P_1 \text{ for } i=1,\dots,n \big\},\notag
\\
X_h&:=V_{t,h} \otimes V_{x,h}, \label{eq:defXh}
\end{align}
it is known, cf.~\cite[Sect.~4]{249.99}, that condition \eqref{infsup} is satisfied for 
\begin{equation}\label{eq:defYh}
Y_h:=\big\{\zeta\in L_2(I)\colon \zeta|_{((i-1)hT,ihT)} \in P_1 \text{ for } i=1,\dots,n \big\}  \otimes V_{x,h},
\end{equation}
where obviously also  $X_h \subset Y_h$.

Applying this approach for $f=(v-v_{0})\partial_x u(v_{0})$, in view of \eqref{quasi} the error of the obtained approximation for $e(v)$ with respect to~the $X$-norm can be bounded 
by the error of the best approximation from $X_h$. To bound the latter error we recall from Theorem~\ref{smoothness} that for $v \in W^1_\infty(I \times \Omega) \cap L_\infty(I;W^2_\infty(\Omega))$, it holds that
$$
e(v)\in \big(H^1_{0,\{0\}}(I) \otimes H^1_0(\Omega) \big)\cap \big(H^2(I) \otimes H^{-1}(\Omega) \big) \cap \big(L_2(I) \otimes H^3(\Omega)\big).
$$

With $Q_{x,h}$, $Q_{t,h}$ denoting the $L_2(\Omega)$- or $L_2(I)$-orthogonal projectors onto $V_{x,h}$ or $V_{t,h}$, respectively, $Q_{t,h} \otimes Q_{x,h}$ is a projector onto $X_h$.
Writing
$$
\identity- Q_{t,h} \otimes Q_{x,h}=( \identity-Q_{t,h}) \otimes Q_{x,h}+\identity \otimes (\identity- Q_{x,h}),
$$
and using that 
\begin{align*}
&\|\identity - Q_{x,h}\|_{\cL(H^1_0(\Omega)\cap H^2(\Omega),H^1_0(\Omega))} \lesssim h,\quad
\|Q_{x,h}\|_{\cL(H_0^1(\Omega),H_0^1(\Omega))} \lesssim 1,\\
&\|\identity - Q_{t,h}\|_{\cL(H^1(I),L_2(I))} \lesssim h, \quad
\|\identity \|_{\cL(L_2(I),L_2(I))} = 1
\end{align*}
by standard interpolation estimates 
 and uniform $H^1$-boundedness of these $L_2$-orthogonal projectors, see e.g.~\cite[\S 3]{34.55}, one infers that
$$
\|\identity- Q_{t,h} \otimes Q_{x,h}\|_{\cL((L_2(I) \otimes (H^1_0(\Omega)\cap H^2(\Omega)) )
\cap ( H^1(I) \otimes H^1_0(\Omega)),L_2(I) \otimes H^1_0(\Omega))} \lesssim h.
$$
Similarly using that
\begin{align*}
&\|\identity - Q_{x,h}\|_{\cL(L_2(\Omega),H^{-1}(\Omega))}=\|\identity - Q_{x,h}\|_{\cL(H^1_0(\Omega),L_2(\Omega))} \lesssim h,\\
&\|Q_{x,h}\|_{\cL(H^{-1}(\Omega),H^{-1}(\Omega))} =\|Q_{x,h}\|_{\cL(H_0^{1}(\Omega),H_0^{1}(\Omega))} \lesssim 1, \\
& \|\identity - Q_{t,h}\|_{\cL(H^2(I),H^1(I))} \lesssim h, \quad  \|\identity\|_{\cL(H^1(I),H^1(I))} =1,
\end{align*}
one infers that
$$
\|\identity- Q_{t,h} \otimes Q_{x,h}\|_{\cL(( H^1(I)  \otimes L_2(\Omega) )\cap ( H^2(I) \otimes H^{-1}(\Omega) ),  H^1(I) \otimes H^{-1}(\Omega))}\lesssim h.
$$

Our findings are summarized in the following theorem.
\begin{theorem}\label{thm:optimal convergence} For $v \in W^1_\infty(I \times \Omega) \cap L_\infty(I;W^2_\infty(\Omega))$ and $X_h$, $Y_h$ as defined in~\eqref{eq:defXh} and~\eqref{eq:defYh}, the numerical approximation $e_h=e_h(v) \in X_h$ to $e=e(v)$ obtained by the application of the minimal residual method to \eqref{R3}\footnote{If necessary taking into account the transformations discussed in Remark~\ref{opmerking}.} satisfies
$$
\|e-e_h\|_X \lesssim h.
$$ 
\end{theorem}

Notice that for this space $X_h$ of continuous piecewise bilinears, this linear decay of the error $\|e-e_h\|_X$ as function of $h$ is generally the best that can be expected. In view of the order of the space $X_h$, one may hope that $\|e-e_h\|_{L_2(I \times \Omega)}$ is ${\mathcal O}(h^2)$, but on the basis of the smoothness demonstrated for $e$, even for $\inf_{\bar{e} \in X_h}\|e-\bar{e}\|_{L_2(I \times \Omega)}$ this cannot be shown.

\section{Interpolation for parametrized drift, boundaries, and final time}
\label{sec:parameter dependence}
In this section we consider the case that $v$ and $T$ in \eqref{R1} depend on a number of parameters $(\rho_1,\ldots,\rho_N) \in [-1,1]^N$, and that one is interested in the solution $u(v)$ to~\eqref{R1} for multiple values of these parameters. As 
explained in Section~\ref{sec:reganalysis}, in order to find $u(v)$ it suffices to obtain the solution $e(v)$ to~\eqref{R3}. Instead of simply solving $e(v)$ for each of the desired parameter values, under the provision that $v$ and $T$ depend smoothly on the parameters, one may attempt to \emph{interpolate} $e(v)$ from its a priori computed approximations for a carefully selected set of parameters in $[-1,1]^N$.

In order to be able to do so, first of all we need to get rid of the parameter dependence of the domain $I \times \Omega=(0,T) \times (0,1)$.
With $\hat{I}:=(0,1)$, the function $\hat{u}$ on $\hat{I}\times \Omega$ defined by 
$\hat{u}(t,x):=u(t T,x)$ solves 
\begin{align}\label{eq:hat u}
\left\{\begin{alignedat}{2}
\partial_t \hat{u}(t, x)&=T[\partial^2_x \hat{u}(t, x) +\hat{v}(t,x) \partial_x \hat{u}(t, x)]&& \quad(t, x) \in \hat{I} \times D,\\
\hat{u}(t,0)&=1, \quad \hat{u}(t,1)=0 &&\quad t \in \hat{I},\\
\hat{u}(0, x)&=0&&\quad   x \in D,
\end{alignedat}
\right.
\end{align}
where analogously $\hat{v}(t,x):=v(t T,x)$. Denoting this $\hat{u}$ as $\hat{u}(\hat{v},T)$, the difference
 $$
 \hat{e}=\hat{e}(\hat{v},T):=\hat{u}(\hat{v},T)-\hat{u}(v_{0},T)\colon (t,x) \mapsto e(tT,x)
 $$
 solves 
\begin{align}\label{eq:final system}
\left\{
\begin{alignedat}{2}
\partial_t \hat{e}(t, x)&= T[\partial^2_x \hat{e}(t, x) 
+ \hat{v}(t,x) \partial_x \hat{e}(t, x)] \\
&\quad+T(\hat{v}(t,x)-v_{0}) \partial_x \hat{u}(v_{0},T)&& \quad(t, x) \in \hat{I} \times \Omega,\\
\hat{e}(t,0)&=0, \quad \hat{e}(t,1)=0 &&\quad t \in \hat{I},\\
\hat{e}(0, x)&=0&&\quad   x \in \Omega.
\end{alignedat}
\right.
\end{align}
By simply replacing $I=(0,T)$ by $\hat{I}=(0,1)$ and in particular $X$ as well as $Y$ by
$$
\hat X:=L_2(\hat I;H^1_0(\Omega)) \cap H^1(\hat I;H^{-1}(\Omega)), \quad \hat Y:=L_2(\hat I;H^1_0(\Omega)),
$$
in a number of places, it is clear that the results that we obtained about the smoothness of $e$ and its numerical approximation $e_h$ by the minimal residual method apply equally well to $\hat{e}$ and its minimal residual approximation that we denote as $\hat{e}_h$.

Since the domain of $\hat{e}$ is independent of parameters, we can apply the idea of interpolation. 
One option is to perform a `full' tensor product interpolation. In this case, the number of \res{interpolation} points required for a fixed polynomial degree, i.e., the number of values of the parameters for which a numerical approximation for $\hat{e} \in \hat{X}$ has to be computed, grows exponentially with the number $N$ of parameters. As this is undesirable, we instead apply a sparse tensor product interpolation. More specifically, we choose the Smolyak construction, based on Clenshaw--Curtis abscissae in each parameter direction, see \cite{239.18}: For $i\in\N$ let $I_{i+1}$ denote the univariate interpolation operator with abscissae $\cos j 2^{-i} \pi $, $j=0,\ldots,2^i$, onto the space of polynomials of degree $2^i$, let \res{$I_1$} be the interpolation operator with abscissa $0$ and let $I_{0}:=0$. Then, for an integer $q\ge N$, we apply the sparse interpolator
$$
{\mathcal I}_q:=\sum_{\{{\bf i} \in \N_0^N\colon \sum_{n=1}^N i_n \leq q\}} \bigotimes_{n=1}^N (I_{i_n}-I_{i_n-1}).
$$
It is known that the resulting interpolation error in $C([-1,1]^N;\hat{X})$ (for arbitrary Banach space $\hat X$), equipped with $\|\cdot\|_{L_\infty([-1,1]^N;\hat{X})}$,  decays subexponentially in the number of interpolation points when $\hat{e}$ as function of each of the parameters $\rho_n$ has an extension to a differentiable mapping on a neighbourhood $\Sigma$ of $[-1,1]$ in $\C$. For details about this statement we refer to \cite[Thm.~3.11]{239.18}.
\cite{239.18} also mentions that the result requires relatively large values of $q$. 
Thus, the authors additionally prove  algebraic convergence under the same assumptions but for arbitrary $q$~\cite[Thm.~3.10]{239.18}.

Instead of $\hat{e}$, we interpolate a numerical approximation $\hat{e}_h$, specifically the one obtained by the minimal residual method described in Section~\ref{SMRM}. 
For the additional error we have
$$
\|{\mathcal I}_q(\hat{e}-\hat{e}_h)\|_{L_\infty([-1,1]^N;\hat{X})} \leq
\|{\mathcal I}_q \|_{\cL(C([-1,1]^N),C([-1,1]^N))} \|\hat{e}-\hat{e}_h\|_{L_\infty([-1,1]^N;\hat{X})}.
$$
In \cite[Sect.~5.3]{35.945} it has been shown that the factor $\|{\mathcal I}_q\|_{\cL(C([-1,1]^N),C([-1,1]^N))}$, known as the Lebesgue constant, is bounded by $(\# \{{\bf i} \in \N_0^N\colon \sum_{n=1}^N i_n \leq q\})^2$, which is only of polylogarithmic  order as function of the number of interpolation points.

Concerning the factor $\|\hat{e}-\hat{e}_h\|_{L_\infty([-1,1]^N;\hat{X})}$, in our derivation of Theorem~\ref{thm:optimal convergence} we have seen that for each parameter value $(\rho_1,\ldots,\rho_N) \in [-1,1]^N$ the expression $h^{-1}\|\hat{e}-\hat{e}_h\|_{\hat{X}}$ can be bounded by a constant multiple, only dependent on an upper bound for $\|\hat{v}\|_{L_\infty(\hat{I} \times \Omega)}$ and for the norm of $\hat{e}$ in $H^1_{0,\{0\}}(\hat I;H^1_0(\Omega))\cap H^2(\hat I;H^{-1}(\Omega)) \cap L_2(\hat I;H^2(\Omega))$. 
For uniformly bounded $T$ and $T^{-1}$, and $\hat{v}$ that varies over a bounded set in $W^1_\infty(\hat{I} \times \Omega) \cap L_\infty(\hat{I};W^2_\infty(\Omega))$, inspection of the estimates from Sect.~\ref{sec:reganalysis} 
reveals that the latter norm of $\hat{e}$ is uniformly bounded. So assuming that these conditions on $T$, $T^{-1}$ and $v$ hold true for 
$(\rho_1,\ldots,\rho_N) \in [-1,1]^N$, we have that
$\|\hat{e}-\hat{e}_h\|_{L_\infty([-1,1]^N;\hat{X})} \lesssim h$.

What remains is to establish the differentiability of the solution $\hat{e}$ as function of each of the parameters which is done in the following theorem.


\begin{theorem}\label{thm:analytic_in_para} For an open $[-1,1] \subset \Sigma \subset \C$, let $(\hat{v},T)\colon \Sigma \rightarrow C(\overline{\hat{I}};W^1_\infty(\Omega)) \times (0,\infty)$ be differentiable. For $\rho\in \Sigma$ let $\hat{e}(\hat{v}(\rho),T(\rho))\in \hat{X}$ be the solution to~\eqref{eq:final system}. Then $\rho \mapsto \hat{e}=\hat{e}(\hat{v}(\rho),T(\rho))\colon \Sigma \rightarrow \hat{X}$ is differentiable.
\end{theorem}

\begin{proof} The proof is based on the fact that $\hat{e}$ is the solution of a well-posed PDE with coefficients and a forcing term that are differentiable functions of $\rho$.

Analogously to \eqref{defL}, denoting by $L(\hat{v},T)$ the map $w \mapsto f$ defined by $\partial_t w=T(\partial_x^2+\hat{v} \partial_x) w+f$ on $\hat{I} \times \Omega$, $w(t,0)=0=w(t,1)$ ($t \in \hat{I}$), and $w(0,x)=0$ ($x \in \Omega$), one has
\begin{equation}\label{eq:hateexp}
\hat{e}(\hat{v}(\rho),T(\rho))=L(\hat{v}(\rho),T(\rho))^{-1} T(\rho)(\hat{v}(\rho)-v_{0}(\rho))\partial_x \hat{u}(v_{0}(\rho),T(\rho)),
\end{equation}
where $v_{0}(\rho):=\hat{v}(\rho)(0,0)$. 
Below we demonstrate that
\
\begin{align} \label{e1}
&\rho \mapsto T(\rho)(\hat{v}(\rho)-v_{0}(\rho))\colon \Sigma \rightarrow L_\infty(\hat{I};W_\infty^1(\Omega)) \text{ is differentiable,}\\  \label{e2}
&\rho \mapsto \hat{u}(v_{0}(\rho),T(\rho))\colon \Sigma  \rightarrow L_2(\hat{I} \times \Omega) \text{ is differentiable,}
\intertext{so that, from $\partial_x \in \cL(L_2(\hat{I}\times \Omega),\hat{Y}')$ and $L_\infty(\hat{I};W_\infty^1(\Omega))$-functions being pointwise multipliers in $\cL(\hat{Y}',\hat{Y}')$,}
  \label{e33}
& \rho \mapsto T(\rho) (\hat{v}(\rho)-v_{0}(\rho))\partial_x \hat{u}(v_{0}(\rho),T(\rho))\colon  \Sigma  \rightarrow \hat{Y}'  \text{ is differentiable.}
\intertext{We proceed below to show that}  \label{e4}
&\rho \mapsto L(\hat{v}(\rho),T(\rho))^{-1}\colon  \Sigma  \rightarrow \cL(\hat{Y}',\hat{X}) \text{ is differentiable.}
\end{align}
Together, \eqref{e33} and  \eqref{e4} complete the proof.

From $\rho \mapsto \hat{v}(\rho) \colon \Sigma \rightarrow C(\overline{\res{\hat{I}}};W^1_\infty(\Omega))$ being differentiable, it follows that 
$\rho \mapsto v_{0}(\rho) \colon \Sigma \rightarrow \C$  is differentiable, which together with $T\colon \Sigma \rightarrow (0,\infty)$ being differentiable shows \eqref{e1}.

To show \eqref{e4}, we fix some arbitrary $\rho_0\in\Sigma$, abbreviate $L:=L(\hat{v}(\rho),T(\rho))$ as well as $L_0:=L(\hat{v}(\rho_0),T(\rho_0))$ and write
\begin{align*}
L^{-1}=L_0^{-1}+L_0^{-1}[L_0-L] L_0^{-1}+L^{-1}\{[L_0-L] L_0^{-1}\}^2.
\end{align*}
This decomposition and the fact that $L(\hat{v}(\rho),T(\rho))^{-1}$ is bounded in $\cL(\hat{Y}',\hat{X})$ for $\rho$ in a neighbourhood of $\rho_0$ (\cite[Thm.~5.1]{247.15}) imply that
it suffices to show that for some $K(\rho_0) \in \cL(\C,\cL(\hat{X},\hat{Y}'))$, 
\be \label{intermediate}
L(\hat{v}(\rho_0),T(\rho_0))-L(\hat{v}(\rho),T(\rho))=K(\rho_0) (\rho-\rho_0)+o(\rho-\rho_0) \text{ in }\cL(\hat{X},\hat{Y}').
\ee

We have
\begin{align*}
&L(\hat{v}(\rho_0),T(\rho_0))-L(\hat{v}(\rho),T(\rho))\\
&=[T(\rho)-T(\rho_0)]\partial_x^2
+ [(T(\rho)-T(\rho_0))\hat{v}(\rho)+T(\rho_0)(\hat{v}(\rho)-\hat{v}(\rho_0))]\partial_x.
\end{align*}
From $T(\rho)-T(\rho_0)=\res{DT}(\rho_0)(\rho-\rho_0)+o(\rho-\rho_0)$, 
$\hat{v}(\rho)-\hat{v}(\rho_0)=\res{D\hat{v}}(\rho_0)(\rho-\rho_0)+o(\rho-\rho_0)$ in $C(\overline{I}_1,W_\infty^1(\Omega)) \hookrightarrow L_\infty(\hat{I} \times \Omega)$,
$\partial_x^2 \in \cL(\hat{X},\hat{Y}')$, 
$\partial_x \in \cL(\hat{X},L_2(\hat{I} \times \Omega))$,
 $L_\infty(\hat{I} \times \Omega)$-functions being pointwise multipliers in $\cL(L_2(\hat{I} \times \Omega),L_2(\hat{I} \times \Omega))$,
and $L_2(\hat{I} \times \Omega) \hookrightarrow \hat{Y}'$,
one concludes \eqref{intermediate}, and so \eqref{e4}.

To show \eqref{e2}, i.e., differentiability of $\rho \mapsto \hat{u}(v_{0}(\rho),T(\rho))$, we repeat the argument that led to~\eqref{eq:hateexp} to obtain
\begin{align*}
\hat{u}(v_{0}(\rho),T(\rho))&=\hat{u}(0,T(\rho))+\hat{u}(v_{0}(\rho),T(\rho))-\hat{u}(0,T(\rho))\\
&=\hat{u}(0,T(\rho))+T(\rho)v_{0}(\rho) L(v_{0}(\rho),T(\rho))^{-1}  \partial_x \hat{u}(0,T(\rho)),
\end{align*}
and show that
\be
\label{50}
\rho \mapsto \hat{u}(0,T(\rho))\colon \Sigma \mapsto L_2(\hat{I} \times \Omega) \text{ is differentiable.}
\ee
Then $\rho \mapsto \partial_x\hat{u}(0,T(\rho))\colon \Sigma \mapsto \hat{Y}'$ is differentiable, and from both
$\rho \mapsto T(\rho)v_{0}(\rho)\colon$ $\Sigma \rightarrow \C$ 
and $\rho \mapsto L(v_{0}(\rho),T(\rho))^{-1}\colon  \Sigma  \rightarrow \cL(\hat{Y}',\hat{X})$ being differentiable one infers \eqref{e2}.

To show \eqref{50}, we apply our approach for the third time.
Picking some $\bar{\rho} \in \Sigma$, we write
$$
\hat{u}(0,T(\rho))=\hat{u}(0,T(\bar{\rho} ))+(T(\bar{\rho})-T(\rho))L(0,T(\rho))^{-1} \partial_x^2 \hat{u}(0,T(\bar{\rho} )).
$$
Knowing that $\partial_x^2 \hat{u}(0,T(\bar{\rho} )) \in \hat{Y}'$, and 
$\rho \mapsto L(0,T(\rho))^{-1}\colon  \Sigma  \rightarrow \cL(\hat{Y}',\hat{X})$
and $\rho \mapsto T(\rho) \colon  \Sigma  \rightarrow \C$ are differentiable, the proof of \eqref{50} and thus of the theorem is completed.
\end{proof}

\section{Numerical results}\label{sec:examples}
We consider three relevant examples of the form \eqref{eq:FPintro} (or its equivalent reformulation~\eqref{R0}) with $\sigma=1$ from the literature. 
We transform the solution $\tilde{u}$ of \eqref{R0}, which might live on a time-dependent spatial domain, to $u$, which satisfies~\eqref{R1} on the domain $(0,T) \times (0,1)$. 
In each example the resulting drift function $v$ as well as the end time point $T$ depend on an up to $N=5$ dimensional parameter ${\bm \rho} \in [-1,1]^N$.


As $u(v({\bm \rho})(0,0),T({\bm \rho}))$ can be computed efficiently as a truncated series, it suffices to consider the difference 
$$e(v({\bm \rho}),T({\bm \rho}))=u(v({\bm \rho}),T({\bm \rho}))-u(v({\bm \rho})(0,0),T({\bm \rho})),$$ 
which satisfies equation~\eqref{R3} and is provably smoother than $u$ (Theorem~\ref{smoothness}).

Thinking of a multi-query setting, instead of approximating this difference for each individual parameter value of interest 
we want to use (sparse) interpolation in the parameter domain $[-1,1]^N$.
To that end, defining $\hat e(t,x):=e(tT({\bm\rho}),x)$, we get rid of the parameter-dependent domain $(0,T({\bm \rho}))\times \Omega$ on which $e$ lives. 
This function $\hat e(t,x)$ satisfies the parabolic problem equation~\eqref{eq:final system} on the space-time domain $\hat{I}\times \Omega=(0,1)^2$ with forcing term
\begin{align*}
\bar w &\mapsto \int_0^1\int_D (\hat v(t,x) - v_{0})\partial_x \hat u(v_{0})(t,x) \bar w(t,x) \,dx\,dt\\
&=\int_0^1\int_D \hat u(v_{0})(t,x)  \big(-\partial_x\hat v(t,x) \bar w(t,x)  - (\hat v(t,x)-v_0) \partial_x\bar w(t,x)\big)\,dx\,dt
\end{align*}
for all $\bar w\in\hat X=L_2(\hat I;H^1_0(\Omega)) \cap H^1(\hat I;H^{-1}(\Omega))$, and
$v_0:=v({\bm \rho})(0,0)$ and corresponding $\hat u(v_0)$ solving \eqref{eq:hat u} with $\hat v = v_0$.

For all sparse interpolation points, by applying the minimal residual method from Section~\ref{SMRM} 
we approximate $\hat e$ by the continuous piecewise affine function $\hat e_h$ on a uniform tensor mesh with mesh-size $h$, where  $\hat u(v_{0})$ inside the forcing term can be efficiently approximated at high accuracy as a truncated series.  

Finally, for all parameter values $\bm \rho$ of interest, we apply the sparse tensor product interpolation analyzed in Section~\ref{sec:parameter dependence} giving rise to an overall error 
\begin{align*}
\|\hat e - \mathcal{I}_q\hat e_h\|_{\hat X} \le \|\hat e - \hat e_h\|_{\hat X} + \|\hat e_h - \mathcal{I}_q \hat e_h\|_{\hat X} 
\approx \|\hat e_{h/2} - \hat e_h\|_{\hat X} + \|\hat e_h - \mathcal{I}_q \hat e_h\|_{\hat X}
\end{align*}
with $q$ the parameter that steers the accuracy of the sparse interpolation.
For each of the considered three examples, we compute the latter two errors for different $h$ and $q$ and parameter test set
\be \label{test-set}
\bm{\rho} \in \{-1,-0.5,0.5,1\}^N.
\ee

By Theorem~\ref{thm:optimal convergence}, we expect $\|\hat e_{h/2} - \hat e_h\|_{\hat X} = \mathcal{O}(h)$ for the first term. 
Section~\ref{sec:parameter dependence} suggests subexponential convergence of the second term $ \|\hat e_h - \mathcal{I}_q \hat e_h\|_{\hat X}$ as function of the number of interpolation points (this was shown for $ \|\hat e - \mathcal{I}_q \hat e\|_{\hat X}$).
However, we already mentioned there that subexponential convergence is only observed for very high $q$ and in practice one should rather expect algebraic convergence. 

Notice that $\|\cdot\|_{\hat X}$ involves a negative order Sobolev norm. Thus, we compute an equivalent version of $\|\cdot\|_{\hat X}$ for functions in the discrete trial space $\bar w\in\hat X_h\subset \hat X$ (similarly for $\bar w\in\hat X_{h/2}$) (see~\cite[Proof of Thm.~3.1]{249.992}) 
\begin{align}\label{eq:equivalent norm}
\| \bar w \|_{\hat X}^2 \eqsim (\bar{\bm{w}}^{h})^\top (\bm{B}^{h})^\top (\bm{A}^{h})^{-1} \bm{B}^{h} \bar{\bm{w}}^{h} + (\bar{\bm{w}}^{h})^\top \bm{C}^{h} \bar{\bm{w}}^{h}.
\end{align}
Here,  $\bar{\bm{w}}^{h}$ is the coefficient vector of $\bar w$ in the standard nodal basis $\Psi^{h}=\{\psi^h_i\}$, $\bm{B}^{h}$ and $\bm{C}^{h}$ are defined as in \eqref{eq:discrete system} with the standard nodal basis $\Phi^{h}=\{\phi^{h}_i\}$, and $\bm{A}^{h}_{i j}:=\int_I \int_\Omega \partial_x \phi^{h}_j(t,x)\partial_x \phi^{h}_i (t,x)\,d x\,dt$.

\subsection{Time-dependent hyperbolic drift function}\label{sec:hyperbolic}
As in \cite{Churchland2008,Hanks2014}, 
we consider 
$$
\mu(t,x):= \mu_0 + \mu_1 \frac{t}{t+t_0}
$$
from Section~\ref{sec:introduction} 
with parameters $\mu_0,\mu_1\in\R$ and $t_0>0$. 
The left and right boundary are given as 
$$
\alpha(t):=0\quad\text{and}\quad \beta(t):=\beta_0
$$ 
with parameter $\beta_0>0$.
Following  \cite{Churchland2008,Hanks2014}, we particularly consider the following practical ranges: $\mu_0 \in [-1.97,-1.64]$, $\mu_1\in[-2.31,-0.99]$, $t_0\in[0.13,0.40]$, $\beta_0\in[1.38,2.26]$, and $\tau\in[0.1,2.5]$ for the end-time point. 
We have $N=5$ different parameters on which $\tilde v$ and thus $v$ depend. 
After rescaling, the parameter space hence has the form $[-1,1]^5$.

In Figure~\ref{fig:b2_estimation}, we plot the maximal error $\hat e_{h/2} - \hat e_h\approx \hat e - \hat e_h$ measured in the (equivalent) $\hat X$-norm \eqref{eq:equivalent norm} over the test set~\eqref{test-set} for different values of $h$. 
Figure~\ref{fig:b2_interpolation}  depicts the maximal interpolation error $\hat{e}_h - \mathcal{I}_q\hat{e}_h$ over the test set~\eqref{test-set} for different values of $h$ and $q$.

\subsection{Space-dependent linear drift function}\label{sec:linear}
As in \cite{Smith2010}, we consider 
$$
\mu(t,x):= \mu_0 + \mu_1 (\beta_0 - x)
$$
from Section~\ref{sec:introduction} 
with parameters $\beta_0>0$  and $\mu_0, \mu_1 \in\R$. 
The left and right boundary are again given as 
$$
\alpha(t):=0\quad\text{and}\quad \beta(t):=\beta_0.
$$ 
Motivated by \cite{Matzke2009,Smith2010}, we particularly consider the following practical ranges: $\mu_0 \in[-2,2]$, $\mu_1\in[-4,4]$, and $\beta_0 \in[0.5,2]$, and choose the end-time point as $\tau:=2.5$. 
We have $N=3$ different parameters on which $\tilde v$ and thus $v$ depend.  
After rescaling, the parameter space hence has the form $[-1,1]^3$.

In Figure~\ref{fig:b3_estimation}, we plot the maximal error $\hat e_{h/2} - \hat e_h\approx\hat e - \hat e_h$ measured in the (equivalent) $\hat X$-norm \eqref{eq:equivalent norm} over the test set~\eqref{test-set}. 
Figure~\ref{fig:b3_interpolation}  depicts the maximal interpolation error $\hat{e}_h - \mathcal{I}_q\hat{e}_h$ 
over the test set~\eqref{test-set} for different values of $h$ and  $q$.

\subsection{Constant drift function and time-dependent linear spatial domain}\label{sec:moving}
We consider a constant drift function
$$
\mu(t,x):= \mu_0
$$
with parameter $\mu_0\in\R$. 
As in \cite{Evans2020} (see also Example~\ref{ex:linear boundaries}), we choose the left and right boundary as 
$$
\alpha(t):=\beta_0\frac{t}{2  T_0}\quad\text{and}\quad \beta(t):=\beta_0 \Big(1-\frac{t}{2 T_0}\Big)
$$ 
with parameters $\beta_0,T_0>0$. Recall from  Example~\ref{ex:linear boundaries} that
$$
  \theta(t) = \frac{ \beta_0^2(T_0-2\widetilde{T})^2 t }
 { T_0^2 - 2\beta_0^2(T_0-2\widetilde{T})t },\quad t\in [0,T) 
$$
 with $T=\theta^{-1}(\widetilde{T})=\frac{T_0 \widetilde{T}}{\beta_0^2 (T_0-2\widetilde{T})}$. 
Following \cite{Evans2020}, we particularly consider the following practical ranges: $\mu_0\in[-5.86,0]$, $\beta_0\in[0.56,3.93]$, $T_0 \in[3,20]$, and $\tau\in [0.1,2.5]$ for the end-time point. 
We have $N=4$ different parameters on which $\tilde v$ and thus $v$ depend. 
After rescaling, the parameter space hence has the form $[-1,1]^4$. 
\res{Figures~\ref{fig:e_hat},~\ref{fig:u_hat}, and~\ref{fig:u_til} show approximations of the solution $\hat{e}$ to~\eqref{eq:final system}, the solution $\hat{u}$ to~\eqref{eq:hat u}, and the solution $\tilde{u}$ to the original problem~\eqref{R0}, with parameter values $\mu_0 = 0$, $\beta_0 = 3.93$, $T_0 = 3$, and $\tau = 2.5$. }
In Figure~\ref{fig:b4_estimation}, we plot the maximal error $\hat e_{h/2} - \hat e_h\approx \hat e - \hat e_h$ measured in the (equivalent) $\hat X$-norm \eqref{eq:equivalent norm}  over the test set~\eqref{test-set}. 
Figure~\ref{fig:b4_interpolation}  depicts the maximal interpolation error $\hat{e}_h - \mathcal{I}_q\hat{e}_h$
 over the test set~\eqref{test-set} for different values of $h$ and $q$.
\begin{figure}[thp] 
\captionsetup{width=0.45\linewidth}
\begin{minipage}[t]{0.5\textwidth}
\begin{center}
\includegraphics[width=0.9\textwidth]{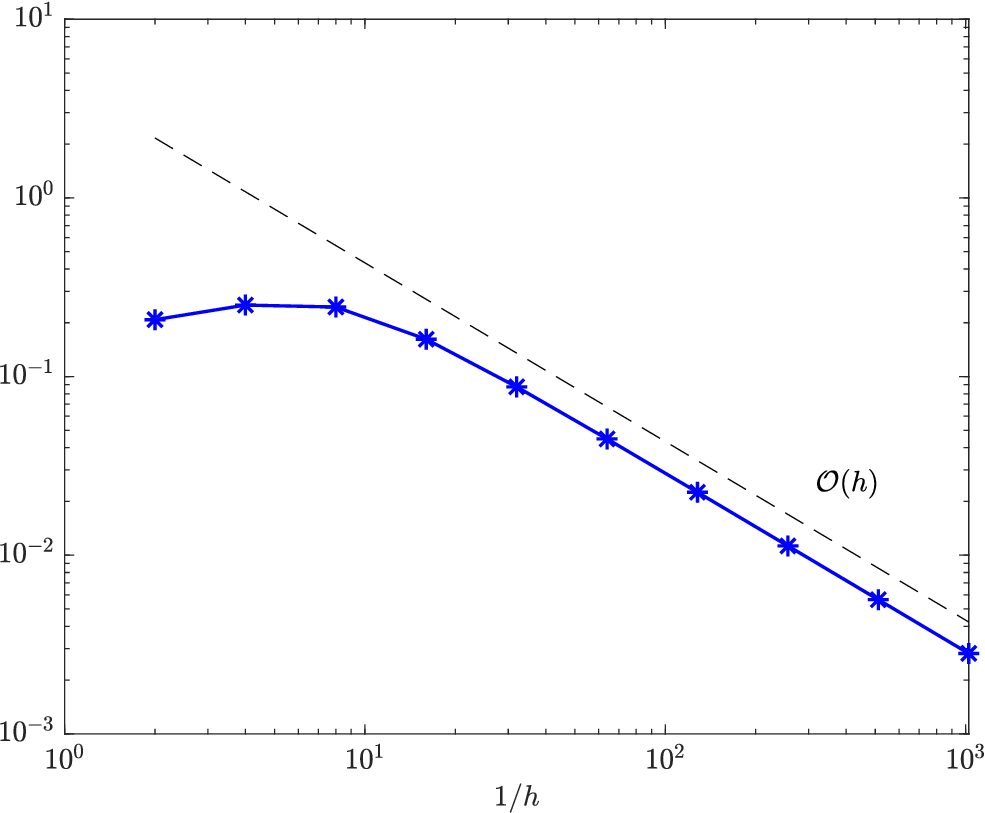}
\caption{\label{fig:b2_estimation}
Maximal error $\hat e_{h/2}(\hat v({\bm \rho})) - \hat e_h(\hat v({\bm \rho}))$ measured in (equivalent) $\hat X$-norm over all 
${\bm \rho}\in\{-1,-0.5,0,0.5,1\}^5$ for time-dependent hyperbolic drift function from Section~\ref{sec:hyperbolic}.
}
\end{center}
\end{minipage}%
\begin{minipage}[t]{0.5\textwidth}
\begin{center}
\includegraphics[width=0.9\textwidth]{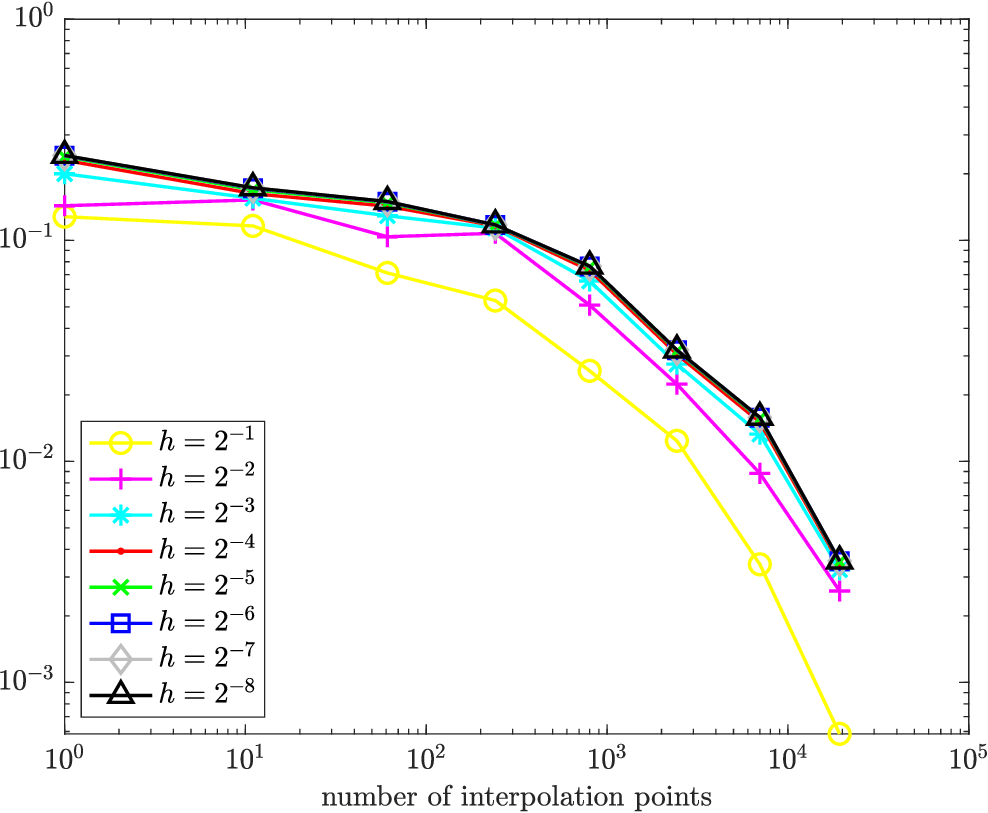}
\caption{\label{fig:b2_interpolation} 
Maximal interpolation error $\hat e_{h}(\hat v({\bm \rho})) - (\mathcal{I}_q\hat e_h(\hat v(\cdot)))({\bm \rho})$ 
for various choices of $h$
measured in (equivalent) $\hat X$-norm over all 
${\bm \rho}\in\{-1,-0.5,0,0.5,1\}^5$ for time-dependent hyperbolic drift function from Section~\ref{sec:hyperbolic}.
}
\end{center}
\end{minipage}
\end{figure}

\begin{figure}[thp] 
\captionsetup{width=0.45\linewidth}
\begin{minipage}[t]{0.5\textwidth}
\begin{center}
\includegraphics[width=0.9\textwidth]{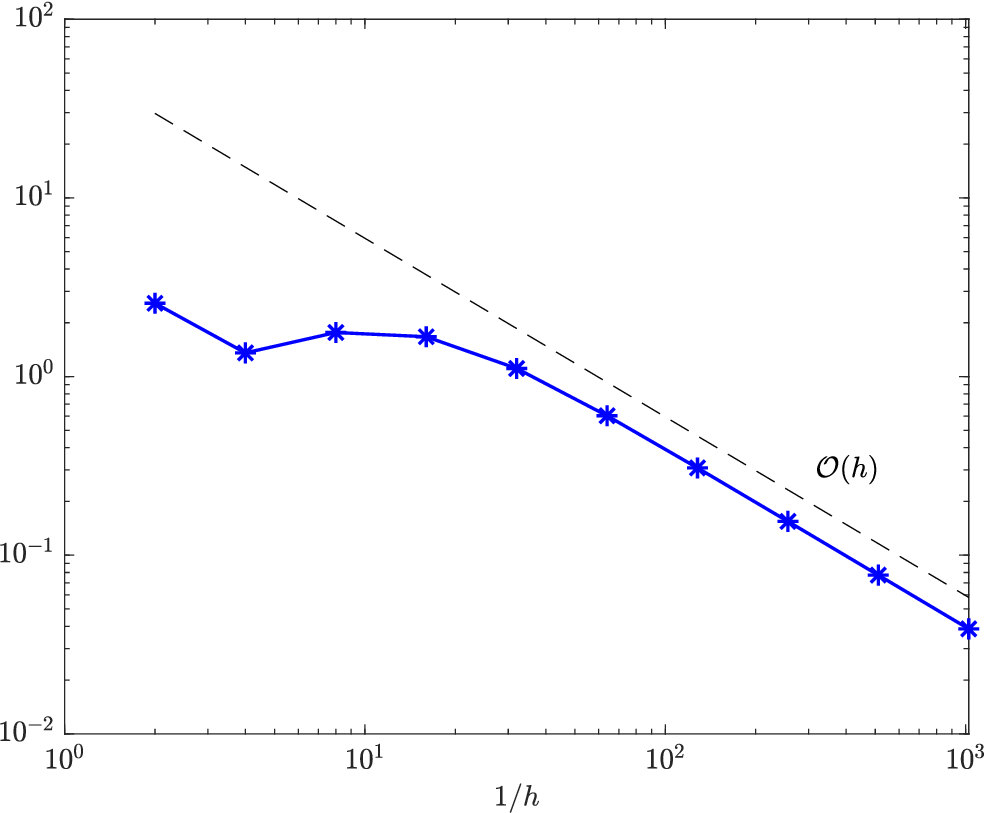}
\caption{\label{fig:b3_estimation}
Maximal error $\hat e_{h/2}(\hat v({\bm \rho})) - \hat e_h(\hat v({\bm \rho}))$ measured in (equivalent) $\hat X$-norm over all 
${\bm \rho}\in\{-1,-0.5,0,0.5,1\}^3$ for space-dependent linear drift function from Section~\ref{sec:linear}.
}
\end{center}
\end{minipage}%
\begin{minipage}[t]{0.5\textwidth}
\begin{center}
\includegraphics[width=0.9\textwidth]{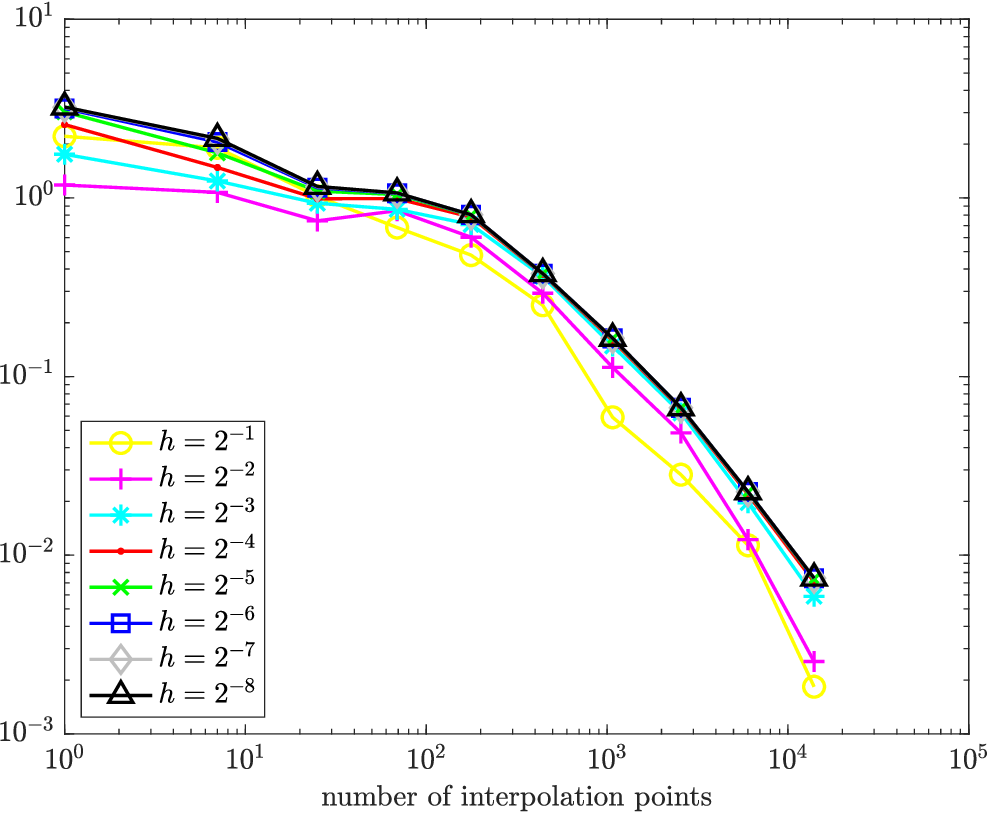}
\caption{\label{fig:b3_interpolation}
Maximal interpolation error $\hat e_{h}(\hat v({\bm \rho})) - (\mathcal{I}_q\hat e_h(\hat v(\cdot)))({\bm \rho})$ with $h$ $=$ $2^{-1},\dots,2^{-8}$ measured in (equivalent) $\hat X$-norm over all 
${\bm \rho}\in\{-1,-0.5,0,0.5,1\}^3$ for space-dependent linear drift function from Section~\ref{sec:linear}.
}
\end{center}
\end{minipage}
\end{figure}
\begin{figure}[thp]
\captionsetup{width=0.45\linewidth}
\begin{minipage}[t]{0.5\textwidth}
\begin{center}\includegraphics[width=0.9\textwidth]{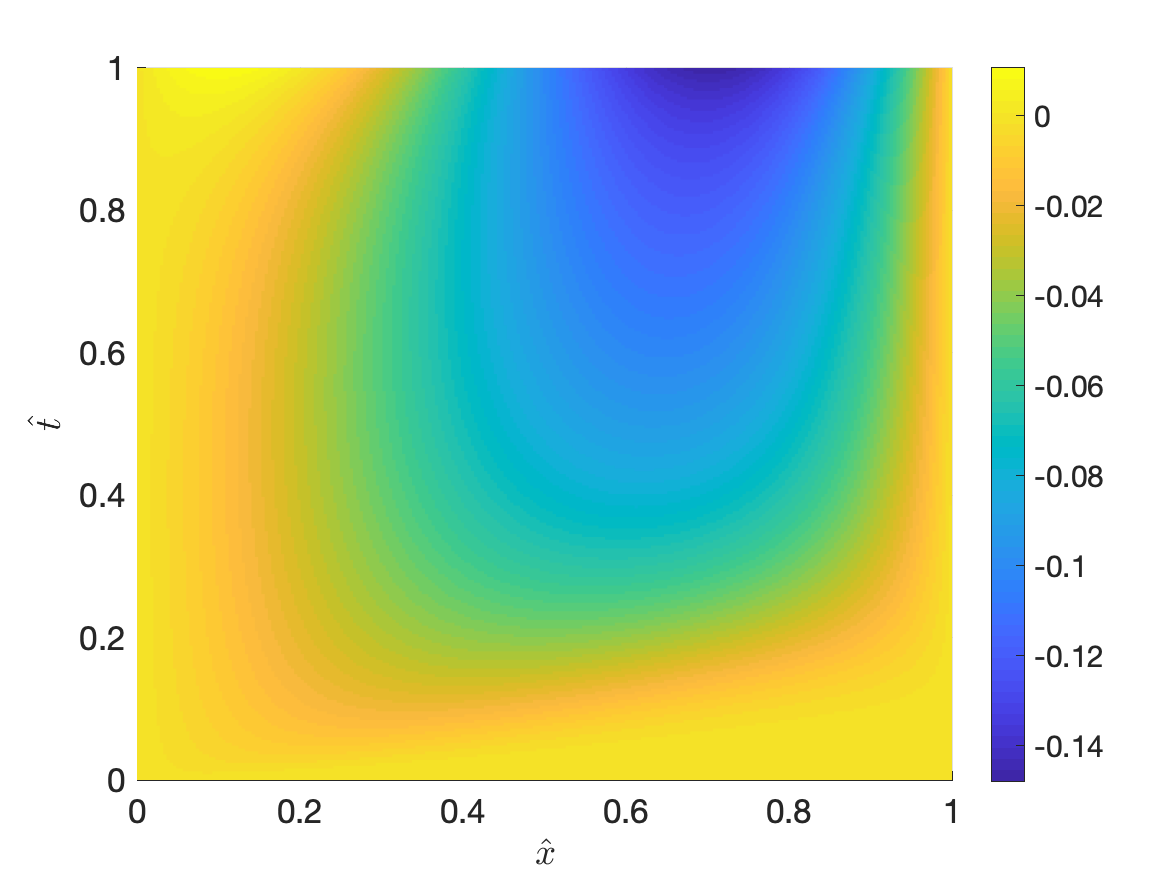}
\caption{\label{fig:e_hat} \res{An approximation of solution $\hat{e}$ to~\eqref{eq:final system} using the minimal residual method.}}
\end{center}
\end{minipage}
\begin{minipage}[t]{0.5\textwidth}
\begin{center}
\includegraphics[width=0.9\textwidth]{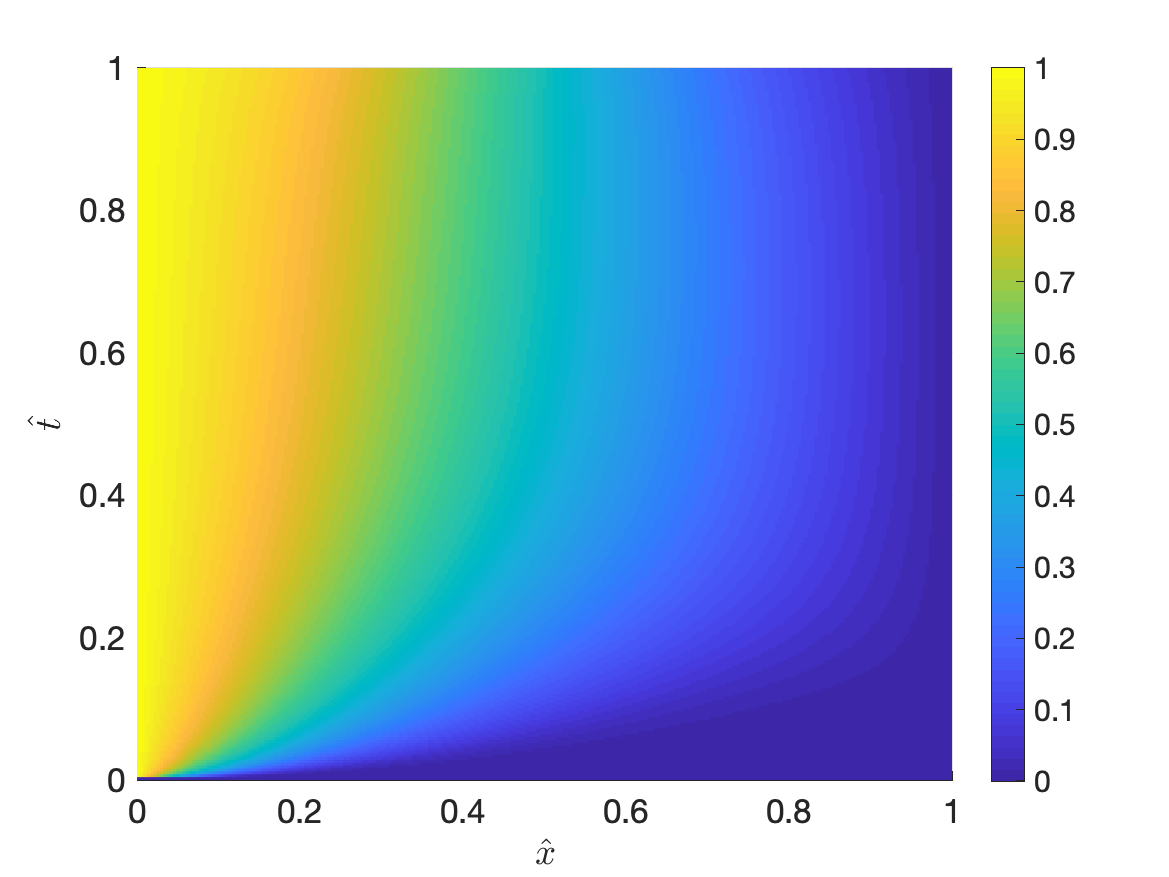}
\caption{\label{fig:u_hat} \res{An approximation of the solution $\hat{u}$ to~\eqref{eq:hat u} obtained by adding $\hat{u}(v_0,T)$ to $\hat{e}$.}}
\end{center}
\end{minipage}
\begin{center}
\begin{minipage}[t]{0.5\textwidth}
\includegraphics[width=0.9\textwidth]{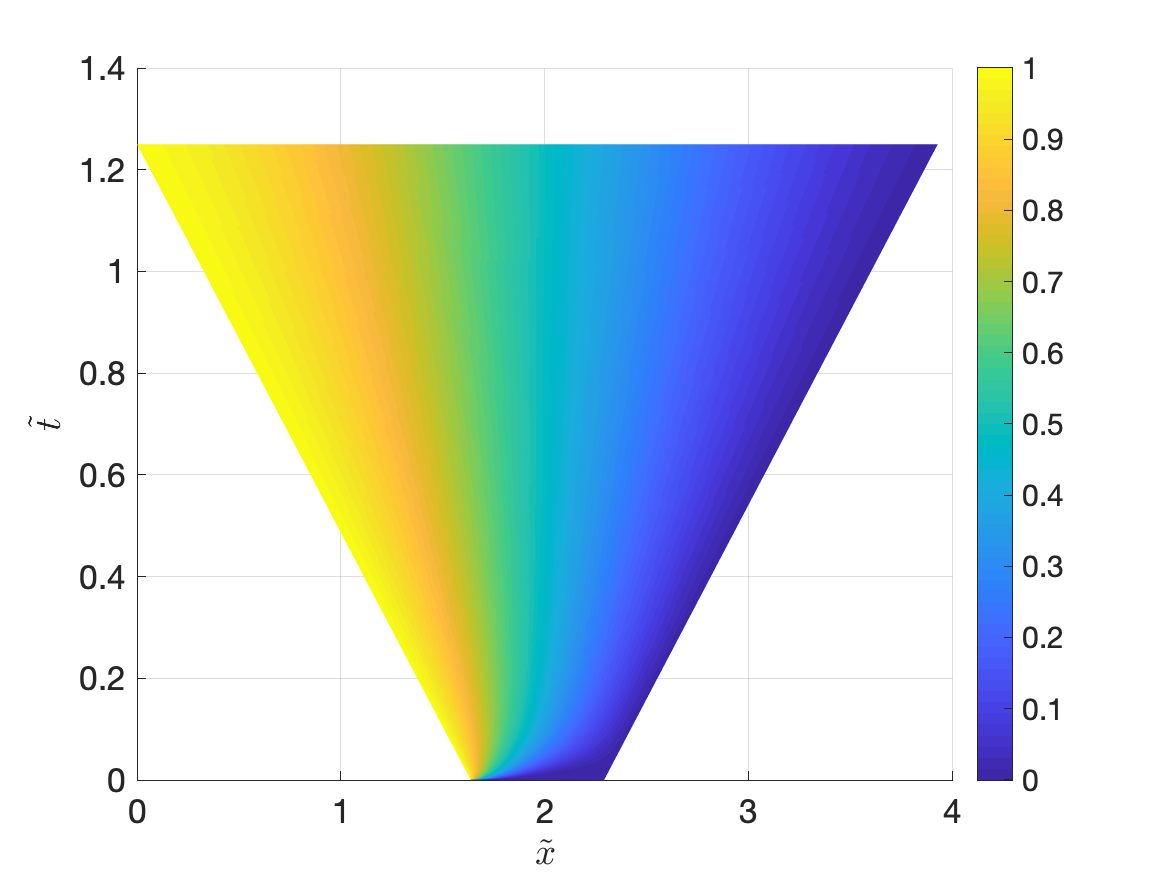}
\caption{\label{fig:u_til} \res{An approximation of the solution $\tilde{u}$ to~\eqref{R0} obtained by transforming the approximation of $\hat{u}$ back to the original domain.}}
\end{minipage}
\end{center}
\end{figure}


\begin{figure}[thp] 
\captionsetup{width=0.45\linewidth}
\begin{minipage}[t]{0.5\textwidth}
\begin{center} 
\includegraphics[width=0.9\textwidth]{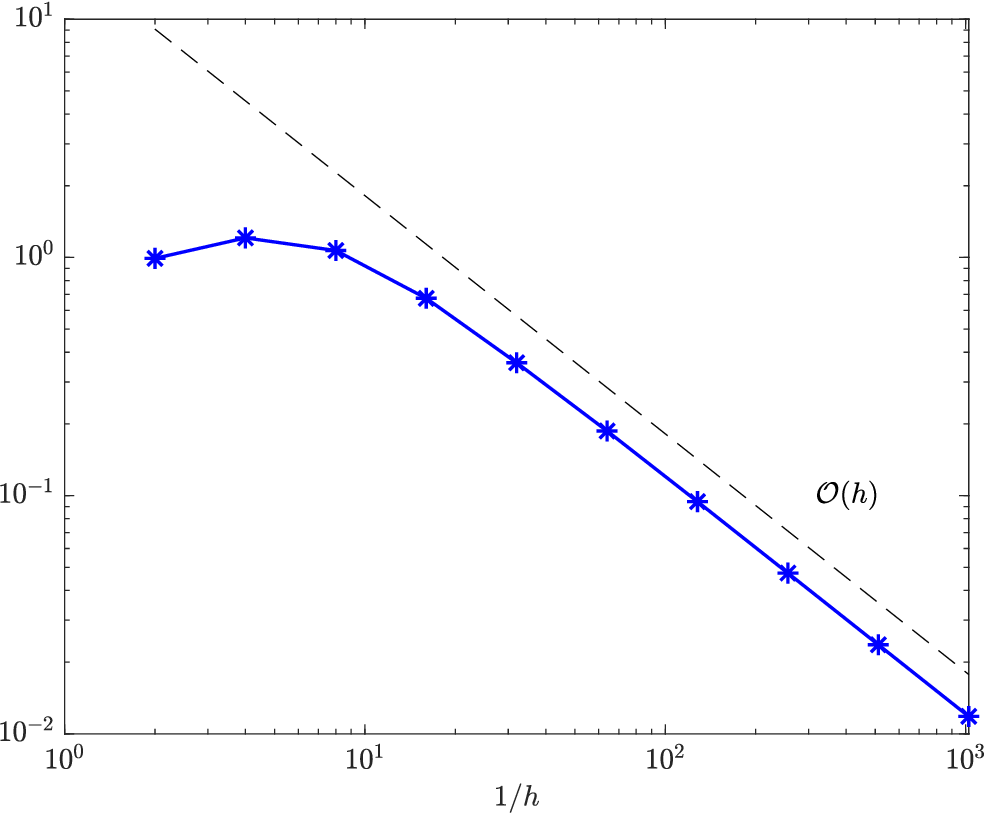}
\caption{\label{fig:b4_estimation}
Maximal error $\hat e_{h/2}(\hat v({\bm \rho}),T({\bm \rho})) - \hat e_h(\hat v({\bm \rho}),T({\bm \rho}))$ measured in (equivalent) $\hat X$-norm over all 
${\bm \rho}\in\{-1,-0.5,0,0.5,1\}^4$ for constant drift function with time-dependent linear spatial domain from Section~\ref{sec:moving}.
}
\end{center}
\end{minipage}%
\begin{minipage}[t]{0.5\textwidth}
\begin{center}
\includegraphics[width=0.9\textwidth]{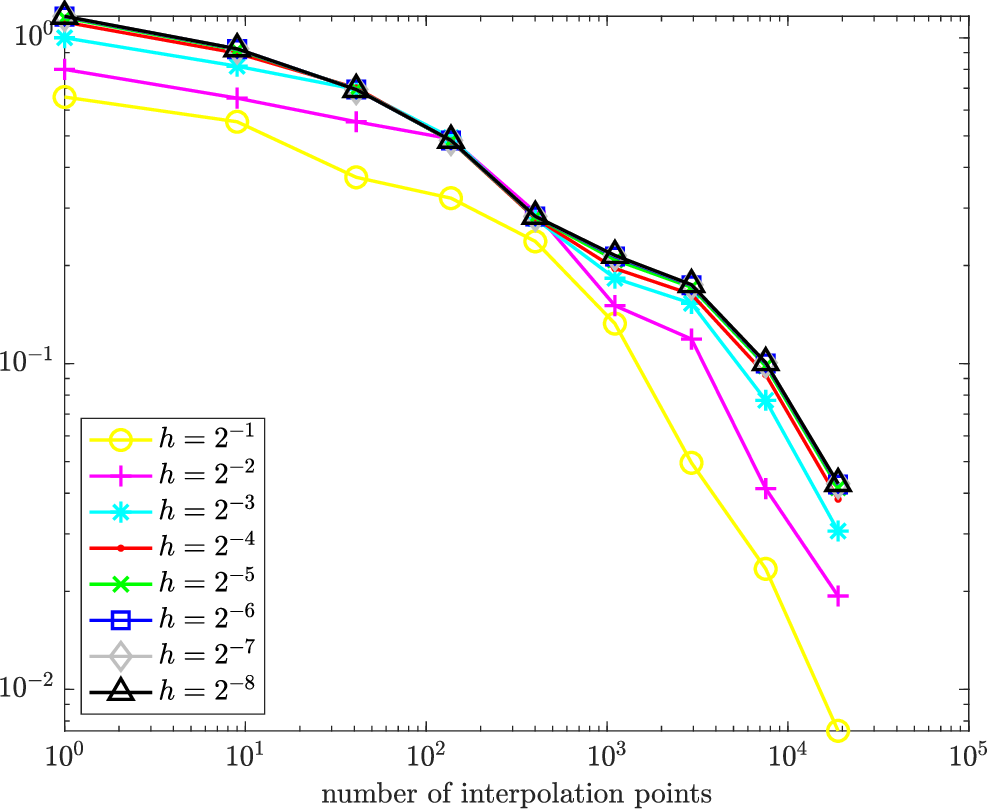}
\caption{\label{fig:b4_interpolation}
Maximal interpolation error $\hat e_{h}(\hat v({\bm \rho}), T({\bm \rho})) - (\mathcal{I}_q\hat e_h(\hat v(\cdot),T(\cdot)))({\bm \rho})$ for various choices of $h$
measured in (equivalent) $\hat X$-norm over all 
${\bm \rho}\in\{-1,-0.5,0,0.5,1\}^4$ for constant drift function with time-dependent linear spatial domain from  Section~\ref{sec:moving}.
}
\end{center}
\end{minipage}

\end{figure}

\section{Conclusion}
We have developed a numerical solution method for solving the Fokker--Planck equation on a one-dimensional spatial domain and with a discontinuity between initial and boundary data and time-dependent boundaries. We first transformed the equation to an equation on a rectangular time-space domain. We then demonstrated that the solution of a corresponding equation with a suitable constant drift function, whose solution is explicitly available as a fast converging series expansion, captures the main singularity present in the solution for a variable  drift function.
The equation for the difference of both these solutions, which is thus more regular than both terms, is solved with a minimal residual method. This method is known to give a quasi-best approximation from the selected trial space. 

Finally, in order to efficiently solve Fokker--Planck equations that depend on multiple parameters, we demonstrate that the solution is a holomorphic function of these parameters. Consequently, a sparse tensor product interpolation method can be shown to converge at a subexponentional rate as function of the number of interpolation points. In one test example, this interpolation method works very satisfactory, but the results are less convincing in two other cases. We envisage that in those cases better results can be obtained by an adaptive sparse interpolation method as the one proposed in \cite{35.946}.

\bibliographystyle{amsalpha}
\bibliography{ref}

\newcommand{\etalchar}[1]{$^{#1}$}
\providecommand{\bysame}{\leavevmode\hbox to3em{\hrulefill}\thinspace}
\providecommand{\MR}{\relax\ifhmode\unskip\space\fi MR }
\providecommand{\MRhref}[2]{%
  \href{http://www.ams.org/mathscinet-getitem?mr=#1}{#2}
}
\providecommand{\href}[2]{#2}
\begin{thebibliography}{AKTSM18}

\bibitem[AKTSM18]{ArtimeEtAl:2018}
O.~Artime, N.~Khalil, R.~Toral, and M.~San~Miguel, \emph{First-passage
  distributions for the one-dimensional {F}okker-{P}lanck equation}, Phys. Rev.
  E \textbf{98} (2018), no.~4, 042143.

\bibitem[And13]{11}
R.~Andreev, \emph{Stability of sparse space-time finite element discretizations
  of linear parabolic evolution equations}, IMA J. Numer. Anal. \textbf{33}
  (2013), no.~1, 242--260.

\bibitem[BCGS21]{BCGS21}
U.~Boehm, S.~Cox, G.~Gantner, and R.~Stevenson, \emph{Fast solutions for the
  first-passage distribution of diffusion models with space-time-dependent
  drift functions and time-dependent boundaries}, J. Math. Psych. \textbf{105}
  (2021), 102613.

\bibitem[BKG12]{Bowman2012}
N.E. Bowman, K.P. Kording, and J.A. Gottfried, \emph{{Temporal integration of
  olfactory perceptual evidence in human orbitofrontal cortex}}, Neuron
  \textbf{75} (2012), 916--927.

\bibitem[BX91]{34.55}
J.H. Bramble and J.~Xu, \emph{Some estimates for a weighted \({L}^2\)
  projection}, Math. Comp. \textbf{56} (1991), 463--476.

\bibitem[CCS14]{35.946}
A.~Chkifa, A.~Cohen, and Ch. Schwab, \emph{High-dimensional adaptive sparse
  polynomial interpolation and applications to parametric {PDE}s}, Found.
  Comput. Math. \textbf{14} (2014), no.~4, 601--633.

\bibitem[Cha43]{Chandrasekhar:1943}
Subrahmanyan Chandrasekhar, \emph{{Dynamical friction. I. General
  considerations: the coefficient of dynamical friction}}, Astrophysical
  Journal \textbf{97} (1943), 255--262.

\bibitem[Chk14]{35.945}
A.~Chkifa, \emph{{Sparse polynomial methods in high dimension: Application to
  parametric PDE}}, Ph.D. thesis, {Universit\'{e} Pierre et Marie Curie - Paris
  VI}, 2014.

\bibitem[CKS08]{Churchland2008}
A.K. Churchland, R.~Kiani, and M.N. Shadlen, \emph{{Decision-making with
  multiple alternatives}}, Nature Neuroscience \textbf{11} (2008), no.~6,
  693--702.

\bibitem[Cos90]{45.491}
M.~Costabel, \emph{Boundary integral operators for the heat equation}, Integral
  Equations Operator Theory \textbf{13} (1990), no.~4, 498--552.

\bibitem[DHP03]{64.53}
R.~Denk, M.~Hieber, and J.~Pr\"{u}ss, \emph{{${\mathcal R}$}-boundedness,
  {F}ourier multipliers and problems of elliptic and parabolic type}, Mem.
  Amer. Math. Soc. \textbf{166} (2003), no.~788, viii+114.

\bibitem[dS64]{64.57}
L.~de~Simon, \emph{Un'applicazione della teoria degli integrali singolari allo
  studio delle equazioni differenziali lineari astratte del primo ordine},
  Rend. Sem. Mat. Univ. Padova \textbf{34} (1964), 205--223.

\bibitem[ETH20]{Evans2020}
N.J. Evans, J.S. Trueblood, and W.R. Holmes, \emph{{A parameter recovery
  assessment of time-variant models of decision-making}}, Behavior Research
  Methods \textbf{52} (2020), 193--206.

\bibitem[FF03]{75.215}
N.~Flyer and B.~Fornberg, \emph{Accurate numerical resolution of transients in
  initial-boundary value problems for the heat equation}, J. Comput. Phys.
  \textbf{184} (2003), no.~2, 526--539.

\bibitem[FFGC21]{Fengler2020}
A.~Fengler, M.~Frank, L.~Govindarajan, and T.~Chen, \emph{{Likelihood
  Approximation Networks (LANs) for Fast Inference of Simulation Models in
  Cognitive Neuroscience}}, Elife \textbf{10} (2021), e65074.

\bibitem[GBK14]{168.84}
M.~Gondan, S.P. Blurton, and M~Kesselmeier, \emph{Even faster and even more
  accurate first-passage time densities and distributions for the {W}iener
  diffusion model}, J. Math. Psych. \textbf{60} (2014), 20--22.

\bibitem[GS01]{Gold2001}
J.~I. Gold and M.~N. Shadlen, \emph{Neural computations that underlie decisions
  about sensory stimuli}, Trends in Cognitive Sciences \textbf{5} (2001),
  no.~1, 10--16.

\bibitem[HFW{\etalchar{+}}15]{Hawkins2015}
G.E. Hawkins, B.U. Forstmann, E.-J. Wagenmakers, R.~Ratcliff, and S.D. Brown,
  \emph{{Revisiting the evidence for collapsing boundaries and urgency signals
  in perceptual decision-making}}, Journal of Neuroscience \textbf{35} (2015),
  no.~6, 2476--2484.

\bibitem[HKS14]{Hanks2014}
T.~Hanks, R.~Kiani, and M.N. Shadlen, \emph{{A neural mechanism of
  speed-accuracy tradeoff in macaque area LIP}}, eLife \textbf{3} (2014),
  e02260.

\bibitem[HS15]{HolcmanSchuss:2015}
D.~Holcman and Z.~Schuss, \emph{Stochastic narrow escape in molecular and
  cellular biology}, vol.~48, Springer, New York, 2015.

\bibitem[MW09]{Matzke2009}
D.~Matzke and E.J. Wagenmakers, \emph{{Psychological interpretation of the
  ex-Gaussian and shifted Wald parameters: A diffusion model analysis.}},
  Psychonomic Bulletin {\&} Review \textbf{16} (2009), no.~5, 798--817.

\bibitem[NTW08]{239.18}
F.~Nobile, R.~Tempone, and C.~G. Webster, \emph{A sparse grid stochastic
  collocation method for partial differential equations with random input
  data}, SIAM J. Numer. Anal. \textbf{46} (2008), no.~5, 2309--2345.

\bibitem[{\O}ks98]{Oksendal1998}
B.~{\O}ksendal, \emph{{Stochastic Differential Equations}}, 5th ed. ed.,
  Springer, Berlin, 1998.

\bibitem[Rat78]{Ratcliff1978}
R.~Ratcliff, \emph{{A theory of memory retrieval}}, Psychological Review
  \textbf{85} (1978), no.~2, 59--108.

\bibitem[SK13]{Shadlen2013}
Michael~N. Shadlen and Roozbeh Kiani, \emph{{Decision making as a window on
  cognition}}, Neuron \textbf{80} (2013), no.~3, 791--806.

\bibitem[Smi10]{Smith2010}
P.L. Smith, \emph{{From Poisson shot noise to the integrated Ornstein-Uhlenbeck
  process: Neurally principled models of information accumulation in
  decision-making and response time.}}, J. Math. Psych. \textbf{54} (2010),
  266--283.

\bibitem[SS09]{247.15}
Ch. Schwab and R.P. Stevenson, \emph{A space-time adaptive wavelet method for
  parabolic evolution problems}, Math. Comp. \textbf{78} (2009), 1293--1318.

\bibitem[SW21a]{249.992}
R.P. Stevenson and J.~Westerdiep, \emph{Minimal residual space-time
  discretizations of parabolic equations: Asymmetric spatial operators},
  Preprint \textbf{arXiv:2106.01090} (2021).

\bibitem[SW21b]{249.99}
\bysame, \emph{Stability of {G}alerkin discretizations of a mixed space-time
  variational formulation of parabolic evolution equations}, {IMA J. Numer.
  Anal.} \textbf{41} (2021), no.~1, 28--47.

\bibitem[VV08]{Voss2008}
A.~Voss and J.~Voss, \emph{A fast numerical algorithm for the estimation of
  diffusion model parameters}, J. Math. Psych. \textbf{52} (2008), no.~52,
  1--9.

\bibitem[Wlo87]{314.91}
J.~Wloka, \emph{Partial differential equations}, Cambridge University Press,
  Cambridge, 1987.

\end{thebibliography}
\end{document}